\documentclass[11pt, reqno]{amsart}

\usepackage{amsfonts, amssymb, amscd}
\usepackage{graphicx}
\usepackage{hyperref}
\usepackage{parskip}
\usepackage{slashed}
\usepackage{fullpage}

\usepackage{color}

\newtheorem{thm}{Theorem}[section]
\newtheorem{cor}[thm]{Corollary}
\newtheorem{lem}[thm]{Lemma}

\newtheorem{prop}[thm]{Proposition}

\theoremstyle{remark}
\newtheorem{rem}[thm]{Remark}

\theoremstyle{definition}
\newtheorem{defin}[thm]{Definition}

\newcommand{\na}{\mathcal{N}(A)}
\newcommand{\sa}{\mathcal{S}(A)}
\newcommand{\ma}{\mathcal{M}(A)}

\newcommand{\R}{\mathbb{R}}

\newcommand{\C}{\mathbb{C}}

\newcommand{\Z}{\mathbb{Z}}

\newcommand{\Grad}{G_{\mathrm{rad}}}

\renewcommand{\Re}{\mathrm{Re}}
\renewcommand{\Im}{\mathrm{Im}}

\DeclareMathOperator{\chg}{chg}

\title{Global wellposedness of the equivariant Chern-Simons-Schr\"odinger equation}

\author[B. Liu]{Baoping Liu}
\address{University of Chicago}
\email{baoping@math.uchicago.edu}
\author[P. Smith]{Paul Smith}
\address{University of California, Berkeley}
\email{smith@math.berkeley.edu}

\thanks{The second author was supported by NSF grant DMS-1103877.}

\begin{document}

\begin{abstract}
  In this article we consider the initial value problem for the $m$-equivariant
  Chern-Simons-Schr\"odinger model in two spatial dimensions with coupling parameter $g \in \R$.
  This is a covariant NLS type problem that is $L^2$-critical.
  We prove that at the critical regularity, for any equivariance index $m \in \Z$, 
  the initial value problem 
  in the defocusing case ($g < 1$) is globally wellposed and the solution scatters.
  The problem is focusing when $g \geq 1$, and in this case we prove that for
  equivariance indices $m \in \Z$, $m \geq 0$, there exist constants $c = c_{m, g}$ such that,
  at the critical regularity, the initial
  value problem is globally wellposed and the solution scatters when the initial data
  $\phi_0 \in L^2$ is $m$-equivariant and satisfies $\| \phi_0 \|_{L^2}^2 < c_{m, g}$. 
  We also show that $\sqrt{c_{m, g}}$
  is equal to the minimum $L^2$ norm of a nontrivial $m$-equivariant standing wave solution.
  In the self-dual $g = 1$ case, we have the exact numerical values $c_{m, 1} = 8\pi(m + 1)$.
\end{abstract}

\maketitle

\tableofcontents

\section{Introduction}

The two-dimensional Chern-Simons-Schr\"odinger system is a nonrelativistic quantum
model describing the dynamics of a large number of particles in the plane
interacting both directly and via a self-generated field. The variables
we use to describe the dynamics are the scalar field $\phi$, describing the particle
system, and the potential $A$, which can be viewed as a real-valued 1-form on $\R^{2+1}$.  
The associated covariant differentiation operators are defined in terms of the potential $A$ as
\begin{equation}
D_\alpha := \partial_\alpha+ i A_\alpha, \quad \quad \alpha = 0, 1, 2
\label{D alpha}
\end{equation}
With this notation, the action integral for the system is
\begin{equation} \label{Lag}
L(A,\phi) = \frac12 \int_{\R^{2+1}}  
\left[ \Im (\bar \phi D_t \phi) + |D_x \phi|^2 -\frac{g}2 |\phi|^4 \right] dx dt +
\frac12 \int_{\R^{2+1}} A \wedge dA
\end{equation}
where $g \in \R$ is a coupling constant.
The Lagrangian is invariant with respect to the transformations
\begin{equation}\label{gauge-freedom}
\phi \mapsto e^{-i \theta} \phi
\quad \quad
A \mapsto A + d \theta
\end{equation}
for compactly supported real-valued functions $\theta(t, x)$.

Computing the Euler-Lagrange equations results in a covariant NLS equation for $\phi$, coupled with
equations giving the field $F = dA$ in terms of $\phi$:
\begin{equation}
\begin{cases}
D_t \phi
&=
i D_\ell D_\ell \phi + i g \lvert \phi \rvert^2 \phi
\\
F_{01}
&= - \Im(\bar{\phi} D_2 \phi)
\\
F_{02}
&=
\Im(\bar{\phi} D_1 \phi)
\\
F_{12}
&=
-\frac{1}{2} \lvert \phi \rvert^2
\end{cases}
\label{CSScom}
\end{equation}
For indices, we use $\alpha = 0$ for the time variable $t$
and $\alpha = 1, 2$ for the spatial variables $x_1, x_2$. When we
wish to exclude the time variable in a certain expression, we
switch from Greek indices to Roman ones. Repeated indices are assumed
to be summed, Greek ones over $\{0, 1, 2\}$, and Roman ones over $\{1, 2\}$.

The system (\ref{CSScom}) is a basic model of Chern-Simons dynamics
\cite{JaPi91, EzHoIw91, EzHoIw91b, JaPi91b}. For further physical
motivation for studying (\ref{CSScom}), see \cite{JaTe81,Jackiw1982, Jackiw1991, MaPaSo91,
Wi90}.

The Chern-Simons-Schr\"odinger system (\ref{CSScom}) 
inherits from \eqref{Lag} the gauge invariance \eqref{gauge-freedom}.
It is also Galilean-invariant and has conserved
\emph{charge}
\begin{equation} \label{charge}
\chg(\phi) := \int_{\R^2} \lvert \phi \rvert^2 dx
\end{equation}
and \emph{energy}
\begin{equation} \label{energy}
E(\phi) :=
\frac12 \int_{\R^2}
\left[ \lvert D_x \phi \rvert^2 - \frac{g}{2} \lvert \phi \rvert^4 \right] dx
\end{equation}
As the scaling symmetry
\[
\phi(t,x)\rightarrow \lambda\phi(\lambda^2t, \lambda x ), \quad
\phi_0(x)\rightarrow \lambda\phi_0(\lambda x); \quad \lambda>0,
\]
preserves the charge of the initial data $\phi_0$, $L^2_x$ is
the critical space for the main evolution equation of (\ref{CSScom}).

In order for \eqref{CSScom} to be a well-posed system, the gauge freedom
\eqref{gauge-freedom} has to be eliminated.
This is achieved by imposing an additional constraint equation.
In the Coulomb gauge, local wellposedness in $H^2$ is established in \cite{BeBoSa95}.  
Also given are conditions ensuring finite-time blowup.
With a regularization argument, \cite{BeBoSa95} demonstrates global existence
(but not uniqueness) in $H^1$ for small $L^2$ data.
Local wellposedness for data small in $H^s$, $s > 0$, is established in \cite{LiSmTa13}
using the heat gauge. 
We refer the reader to \cite[\S 2]{LiSmTa13}
for a comparison of the Coulomb and heat gauges.
At the critical scaling of $L^2$, local existence implies global existence
for small data; it is an open problem to determine whether the 
Chern-Simons-Schr\"odinger system is
wellposed at the critical regularity in any gauge given small but otherwise arbitrary $L^2$ initial data.

The purpose of this article is to establish the global wellposedness of \eqref{CSScom}
for large $L^2$ data in a symmetry-reduced setting, and with respect to the Coulomb gauge.
We provide a brief introduction to these assumptions here and will formalize them in due course.
The Coulomb gauge condition is the requirement that $\nabla \cdot A_x = 0$.
Under this gauge choice, we assume that the wavefunction $\phi$ is equivariant, i.e.,
in polar coordinates $(r, \theta)$ it admits the representation $\phi(t, r, \theta) = e^{i m \theta} u(t, r)$
for some $m \in \Z$ and some radial function $u \in L^\infty_t L^2_x$. 
The integer $m$ we refer to as the degree of equivariance; it is a topological
quantity that is invariant under the flow. The case $m = 0$ corresponds to the radial case.
The natural defocusing range for this problem is $g < 1$, as $H^1$ solutions in this range
necessarily have positive energy. Positivity of the energy for such $g$ is not immediate from its definition \eqref{energy} 
but will be shown to be a consequence of the Bogomol'nyi identity \eqref{Bogo}.

It is convenient to rewrite \eqref{CSScom} as
\begin{equation} \label{CSS}
\begin{cases}
(i \partial_t + \Delta) \phi &= - 2i A_j \partial_j \phi - i \partial_j A_j \phi + A_0 \phi + A^2_x \phi - g |\phi|^2 \phi \\
\partial_t A_1 - \partial_1 A_0 &= - \Im(\bar{\phi} D_2 \phi) \\
\partial_t A_2 - \partial_2 A_0 &= \Im(\bar{\phi} D_1 \phi) \\
\partial_1 A_2 - \partial_2 A_1 &= - \frac12 |\phi|^2
\end{cases}
\end{equation}
We study \eqref{CSS} in the Coulomb gauge, which is the requirement that
\begin{equation} \label{CoulombGauge}
\partial_1 A_1 + \partial_2 A_2 = 0
\end{equation}
Coupling \eqref{CoulombGauge} with the curvature constraints leads to
\[
A_0 = \Delta^{-1} \left[ \partial_1 \Im(\bar{\phi} D_2 \phi) - \partial_2 \Im(\bar{\phi} D_1 \phi) \right], \quad
A_1 = \frac12 \Delta^{-1} \partial_2 |\phi|^2, \quad
A_2 = - \frac12 \Delta^{-1} \partial_1 |\phi|^2
\]
We may rewrite $A_0$ as
\begin{equation} \label{A0-Cartesian}
A_0 = \Im(Q_{12}(\bar{\phi}, \phi)) + \partial_1(A_2 |\phi|^2) - \partial_2(A_1 |\phi|^2)
\end{equation}
where the null form $Q_{12}$ is defined by
\[
Q_{12}(f, g) = \partial_1 f \partial_2 g - \partial_2 f \partial_1 g
\]
The equivariance ansatz suggests using polar coordinates.
In fact, we will take advantage of both Cartesian coordinates and polar coordinates.
Motivated by the transformations
\[
\partial_r = \frac{x_1}{|x|} \partial_1 + \frac{x_2}{|x|} \partial_2, \quad \quad
\partial_\theta = - x_2 \partial_1 + x_1 \partial_2
\]
and
\[
\partial_1 = (\cos \theta) \partial_r - \frac1r (\sin \theta) \partial_\theta, \quad \quad
\partial_2 = (\sin \theta) \partial_r + \frac1r (\cos \theta) \partial_\theta
\]
we introduce
\begin{equation} \label{A1A2->ArAtheta}
A_r = \frac{x_1}{|x|} A_1 + \frac{x_2}{|x|} A_2, \quad \quad
A_\theta = - x_2 A_1 + x_1 A_2
\end{equation}
which are easily seen to satisfy 
\begin{equation} \label{ArAtheta->A1A2}
A_1 = A_r \cos \theta  - \frac1r A_\theta \sin \theta, \quad \quad
A_2 = A_r \sin \theta + \frac1r A_\theta \cos \theta 
\end{equation}
Using these transformations, we may eliminate $A_1, A_2, \partial_1, \partial_2$
in \eqref{CSS} in favor of $A_r, A_\theta, \partial_r, \partial_\theta$.
In particular,
\[
A_j \partial_j = A_r \partial_r + \frac{1}{r^2} A_\theta \partial_\theta, \quad
\partial_j A_j = \partial_r A_r + \frac1r A_r + \frac{1}{r^2} \partial_\theta A_\theta, \quad
A_1^2 + A_2^2 = A_r^2 + \frac{1}{r^2} A_\theta^2
\]
The main evolution equation of \eqref{CSS} therefore admits the representation
\begin{equation} \label{S-Coulomb}
\begin{split}
(i \partial_t + \Delta) \phi 
&= 
- 2i \left( A_r \partial_r + \frac{1}{r^2} A_\theta \partial_\theta \right) \phi 
- i \left( \partial_r A_r + \frac1r A_r + \frac{1}{r^2} \partial_\theta A_\theta \right) \phi
\\
&\quad + A_0 \phi + A_r^2 \phi + \frac{1}{r^2} A_\theta^2 \phi - g |\phi|^2 \phi
\end{split}
\end{equation}
which in more compact form reads
\begin{equation} \label{phievo}
D_t \phi = i \left(D_r^2 + \frac1r D_r + \frac{1}{r^2} D_\theta^2\right) \phi + i g |\phi|^2 \phi
\end{equation}
We also rewrite the $F = dA$ curvature relations in terms of the variables $t, r, \theta$, with
\begin{equation} \label{curvF}
F_{0r} = \partial_t A_r - \partial_r A_0, \quad
F_{0\theta} = \partial_t A_\theta - \partial_\theta A_0, \quad
F_{r\theta} = \partial_r A_\theta - \partial_\theta A_r
\end{equation}
For instance, we have
\[
x_1 (\partial_t A_1 - \partial_1 A_0) + x_2 ( \partial_t A_2 - \partial_2 A_0)
=
- x_1 \Im(\bar{\phi} D_2 \phi) + x_2 \Im(\bar{\phi} D_1 \phi)
\]
which reduces to
\[
r \partial_t A_r - r \partial_r A_0
=
\Im(\bar{\phi} (x_2 D_1 - x_1 D_2) \phi)
\]
so that
\begin{equation} \label{TimeCurvature}
r \left[ \partial_t A_r - \partial_r A_0 \right] =
- \Im(\bar{\phi} \partial_\theta \phi) + A_\theta |\phi|^2 = - \Im(\bar{\phi} D_\theta \phi)
\end{equation}
Similarly, we obtain
\[
\partial_t A_\theta - \partial_\theta A_0 = r \Im(\bar{\phi} D_r \phi)
\]
and
\[
\partial_1 A_2 - \partial_2 A_1 =
\frac1r \partial_r A_\theta - \frac1r \partial_\theta A_r
\]
which implies
\begin{equation} \label{SpatialCurvature}
\partial_r A_\theta - \partial_\theta A_r = -\frac12 |\phi|^2 r
\end{equation}
Therefore we may write \eqref{CSS} equivalently as
\begin{equation} \label{CSSalt}
\begin{cases}
(i \partial_t + \Delta) \phi &=  
- 2i \left( A_r \partial_r + \frac{1}{r^2} A_\theta \partial_\theta \right) \phi 
- i \left( \partial_r A_r + \frac1r A_r + \frac{1}{r^2} \partial_\theta A_\theta \right) \phi \\
&
\quad + A_0 \phi + A_r^2 \phi + \frac{1}{r^2} A_\theta^2 \phi - g |\phi|^2 \phi \\
\partial_t A_r - \partial_r A_0 &= -\frac1r \Im(\bar{\phi} D_\theta \phi) \\
\partial_t A_\theta - \partial_\theta A_0 &= r \Im(\bar{\phi} D_r \phi) \\
\partial_r A_\theta - \partial_\theta A_r &= - \frac12 |\phi|^2 r \\
\end{cases}
\end{equation}
In polar coordinates, the energy \eqref{energy} takes the form
\begin{equation} \label{polarenergy}
E(\phi) = \frac12 \int_0^{2\pi} \int_0^\infty 
\left( |D_r \phi|^2 + \frac{1}{r^2} |D_\theta \phi|^2 - \frac{g}{2} |\phi|^4 \right) r dr d\theta
\end{equation}
Our next simplification is to restrict to equivariant $\phi$.
Our formulation of the equivariant ansatz implicitly assumes that we have chosen the Coulomb gauge condition
\eqref{CoulombGauge},
which in $A_\theta, A_r$ variables takes the form
\begin{equation} \label{Coulpol}
(\frac1r + \partial_r) A_r + \frac{1}{r^2} \partial_\theta A_\theta = 0
\end{equation}
In particular, we assume that $(A, \phi)$ is of the form
\begin{equation} \label{equian}
\phi(t, x) = e^{i m \theta} u(t, r), \quad
A_1(t, x) = -\frac{x_2}{r} v(t, r), \quad
A_2(t, x) = \frac{x_1}{r} v(t, r), \quad
A_0(t, x) = w(t, r)
\end{equation}
The only assumption that we make on $m$ is that $m \in \Z$, and so in particular
we include the radial case $m = 0$.
This ansatz implies that $A_r = 0$ and that $A_\theta$ is a radial function, 
and so \eqref{Coulpol} is satisfied.
Equivariant solutions, of the form \eqref{equian}, are also known as \emph{vortex} solutions,
and appear in related contexts (see, for instance, 
\cite{PaKh86, VeSc86, VeSc86b, JaWe90, ChSp09, ByHuSe12}).
We also make the natural assumption that $A_0$ decays to zero at spatial infinity
(see the proof of Lemma \ref{lem:A0} and the references therein for further discussion of this point).

Next we rewrite the system \eqref{CSSalt} assuming the equivariant ansatz \eqref{equian}.
Thanks to the ansatz, $\partial_\theta \phi = i m \phi$ holds identically, and so we
make this substitution where convenient.
We obtain
\begin{equation} \label{equiCSS}
\begin{cases}
(i \partial_t + \Delta) \phi &=  
\frac{2m}{r^2} A_\theta \phi + A_0 \phi + \frac{1}{r^2} A_\theta^2 \phi - g |\phi|^2 \phi 
\\
\partial_r A_0 &= \frac1r (m + A_\theta) |\phi|^2  \\
\partial_t A_\theta &=  r \Im(\bar{\phi} \partial_r \phi) \\
\partial_r A_\theta &= - \frac12 |\phi|^2 r \\
A_r &= 0
\end{cases}
\end{equation}
\begin{defin}[Equivariant Sobolev spaces] \label{def:equiSobo}
Let $m \in \Z$. For each $s \geq 0$, we define the function space $H^s_m$ to be the
Sobolev space of all functions $f \in H^s_x$ that admit the decomposition 
$f(x) = f(r, \theta) = e^{i m \theta} u(r)$. We also will use the notation $L^2_m = H^0_m$.
\end{defin}

Our first main theorem is the following.
\begin{thm} \label{thm:main}
Let $g < 1$ and $m \in \Z$. 
Then \eqref{equiCSS} is globally wellposed in $L^2_m$, and,
furthermore, solutions scatter both forward and backward in time.
\end{thm}
For our second main theorem, we introduce the notation
$\Z_+$ to denote $\{0, 1, 2, \ldots \}$. 
In this theorem for the coupling constant we take $g = 1$, the so-called ``critical coupling"
or ``self-dual" coupling value.
\begin{thm} \label{thm:main2}
Let $g = 1$ and $m \in \Z_+$. 
Let $\phi_0 \in L^2_m$ with $\chg(\phi_0) < 8\pi (m+1)$.
Then \eqref{equiCSS} is globally wellposed in $L^2_m$ and scatters both forward and backward in time.
\end{thm}
We have a similar statement for the case $g > 1$, though in this case we have not identified
the numerical values of threshold constants. We do show, however, that the threshold constant is related to soliton solutions.
\begin{thm} \label{thm:main3}
Let $g > 1$ and $m \in \Z_+$.
Then there exists a constant $c_{m, g} > 0$ such that if
$\phi_0 \in L^2_m$ with $\chg(\phi_0) < c_{m, g}$, then
\eqref{equiCSS} is globally wellposed in $L^2_m$ and scatters forward and backward in time.
Moreover, the minimum charge of a nontrivial standing wave solution in the class $L^\infty_t L^2_m$
is equal to $c_{m, g}$.
\end{thm}

The $L^4_{t,x}$ norm plays the role of the scattering norm. Our notions of blowup and scattering
are made precise in the remarks preceding Theorem \ref{thm:Cauchy}, which establishes
the Cauchy theory for \eqref{equiCSS} that is attainable using standard perturbative techniques.
For small data, the sign of $g$ plays no role, and indeed Theorem \ref{thm:Cauchy} 
applies to this case. In fact, all results of \S\S \ref{sec:eqCauchy}--\ref{sec:moravir} hold for any
$g \in \R$. It is only starting in \S \ref{sec:noaps} (in particular, Corollary \ref{cor:globap}) 
where the value of $g$ plays a role. The system \eqref{equiCSS} admits solitons
when $g \geq 1$ and $m \in \Z$ is nonnegative, 
and so in this sense $- \infty < g < 1$ is the natural defocusing parameter range.

The challenge is to prove Theorems \ref{thm:main}--\ref{thm:main3} for large data.
The first step is to reduce to special localized solutions.
Bourgain's induction-on-energy method for the
energy-critical NLS revealed the important role played by solutions
simultaneously localized in frequency and space, see \cite{Bo99}.
Kenig and Merle \cite{KeMe06, KeMe08} subsequently streamlined the arguments
reducing one's consideration to such solutions by means of a concentration-compactness argument.
Minimal-mass blowup solutions of the mass-critical NLS are studied in \cite{TaViZh08}.
We adopt a concentration-compactness argument, modeled closely after that of Killip, Tao, and Visan
\cite{KiTaVi09} for the radial 2-d cubic NLS.
Inspiration also comes from the work of Gustafson and Koo \cite{GuKo11}
on radial 2-d Schr\"odinger maps into the unit sphere, 
which, among other things, extends the arguments
of \cite{KiTaVi09} so as to handle a nonlocal term.

\begin{defin} \label{def:ap}
A solution $\phi$ with lifespan $I$ is said to be almost periodic modulo scaling if
there exist a frequency scale function $N : I \to \R^+$
and a compactness modulus function $C : \R^+ \to \R^+$ such that
\[
\int_{|x| \geq C(\eta) / N(t)} |\phi(t, x)|^2 dx \leq \eta
\]
and
\[
\int_{|\xi| \geq C(\eta) N(t)} |\hat{\phi}(t, \xi)|^2 d \xi \leq \eta
\]
for all $t \in I$ and $\eta > 0$.
\end{defin}
Here we have used $\hat{f}(\xi)$ to denote the Fourier transform of $f$ in the spatial variable $x \in \R^2$
only. We sometimes use the notation $\mathcal{F}(f)$ instead of $\hat{f}$.

\begin{rem}
Solutions of \eqref{CSS} are invariant under the symmetry group $G$ introduced
in \cite[Definition 1.6]{KiTaVi09}, which includes the scaling, rotation, translation, and Galilean symmetries
(the action of $G$ on $\phi$ is as specified in \cite{KiTaVi09} and can easily be extended to act on $A$ as well).
The equivariance ansatz \eqref{equian} breaks the translation and Galilean symmetries, leaving
us with scaling and rotational symmetry. This subgroup is denoted by $\Grad$ in
\cite[Definition 1.6]{KiTaVi09}, as its preserves spherical symmetry; in fact, it preserves $m$-equivariance
for any index $m \in \Z$. Because rotational symmetry corresponds
to the action of a compact symmetry group, it may be neglected for our purposes. 
In fact, it plays no role in 
\cite[Definition 1.14]{KiTaVi09}, which defines almost periodicity modulo $G$ and modulo $\Grad$.
\end{rem}

\begin{lem} \label{lem:blowup}
Suppose that the statement of Theorem \ref{thm:main} 
(or \ref{thm:main2}, \ref{thm:main3}) is not true. Then there exists
a critical element, i.e., a maximal-lifespan solution $\phi$ that is almost periodic
modulo scaling and that blows up both forward and backward in time.
Furthermore, this critical element can be taken to be $m$-equivariant.
We can also ensure that the lifespan $I$ and the frequency-scale function $N : I \to \R^+$
match one of the following two scenarios:
\begin{enumerate}
\item
(Self-similar solution) We have $I = (0, +\infty)$ and
\[
N(t) = t^{-1/2} \quad \text{for all } t \in I
\]
\item
(Global solution) We have $I = \R$ and
\[
\sup_{t \in \R} N(t) < \infty \quad \text{for all } t \in I
\]
\end{enumerate}
\end{lem}

Our strategy for proving Theorems \ref{thm:main}--\ref{thm:main3} is to show that the scenarios described in
Lemma \ref{lem:blowup} cannot occur, in the spirit of \cite{KiTaVi09, KiViZh08, GuKo11}.
The first step of the program is to establish that the solutions described by Lemma \ref{lem:blowup}
are special in that they enjoy extra regularity and in particular are in $H^s$ for each $s > 0$.
The energy \eqref{energy} is at the level of $H^1$, and its conservation can be exploited
in both scenarios. To rule out the global profile, we also use a localized virial identity.
This identity can also be adapted to handle the self-similar profile, as described in \cite[\S9]{KiTaVi09},
though we opt instead to rule out the self-similar profile using energy conservation.

The rest of this article is laid out as follows. In the next section, \S \ref{sec:eqCauchy}, we develop the basic
Cauchy theory for \eqref{equiCSS}. Next, in \S \ref{sec:freq}, we introduce the Littlewood-Paley theory
that we will require and we establish how frequency localizations of the nonlinearity $\Lambda(\phi)$, defined in
\eqref{Nphi}, depend upon frequency localizations of input functions $\phi$. Section \ref{sec:reg} establishes
extra regularity for almost periodic solutions, a key technical step in the large data theory.
In \S \ref{sec:moravir}, we establish virial and Morawetz identities. These play an important role
in \S \ref{sec:noaps}, which concludes the proof of Theorem \ref{thm:main} in the $g < 1$ case
by ruling out the blowup scenarios of Lemma \ref{lem:blowup}. In \S \ref{sec:ext}, we consider
the focusing problem, proving Theorems \ref{thm:main2} and \ref{thm:main3} along with some auxiliary results.

\section{The equivariant Cauchy theory} \label{sec:eqCauchy}

Throughout this section we assume that $\phi$ is $m$-equivariant.
A trivial consequence of this that we will repeatedly use is that $|\phi|^2$ is radial.
We assume that all spatial $L^p$ spaces are based on the 2-dimensional Lebesgue measure.

Define the operators $[\partial_r]^{-1}$, $[r^{-n} \bar{\partial}_r]^{-1}$, and $[r \partial_r]^{-1}$ by
\[
\begin{split}
[\partial_r]^{-1} &= - \int_r^\infty f(s) ds, \quad \quad
[r^{-n} \bar{\partial}_r]^{-1} f(r) = \int_0^r f(s) s^n ds \\
[r \partial_r]^{-1} f(r) &= - \int_r^\infty \frac1s f(s) ds
\end{split}
\]
Then straightforward arguments imply
\begin{align}
\| [r \partial_r]^{-1} f \|_{L^p} &\lesssim_p \| f \|_{L^p}, \quad \quad 1 \leq p < \infty \label{direct1} \\
\| r^{-n-1} [r^{-n} \bar{\partial}_r]^{-1} f \|_{L^p} &\lesssim_p \| f \|_{L^p}, \quad \quad 1 < p \leq \infty \label{direct2} \\
\| [\partial_r]^{-1} f \|_{L^2} &\lesssim \| f \|_{L^1} \label{direct3}
\end{align}

\begin{lem}[Bounds on $A_\theta$ terms] \label{lem:Atheta}
We have
\begin{align}
\| A_\theta \|_{L^\infty_{x}} &\lesssim \| \phi \|_{L^2_x}^2 \label{Athet1} \\
\| \frac1r A_\theta \|_{L^\infty_x} &\lesssim \| \phi \|_{L^4_{x}}^2 \label{Athet2} \\
\| \frac{A_\theta}{r^2} \|_{L^p_x} &\lesssim \| \phi \|_{L^{2p}_x}^2, \quad 1 < p \leq \infty \label{Athet3}
\end{align}
\end{lem}
\begin{proof}
We start with
\begin{equation} \label{Atheta}
A_\theta = -\frac12 \int_0^r |\phi|^2 s ds
\end{equation}
which we obtain by integrating the $F_{r \theta}$ spatial curvature condition in \eqref{equiCSS}
($F_{r \theta}$ is given in \eqref{curvF} and simplifies under \eqref{equian}).
To justify the boundary condition, note that \eqref{A1A2->ArAtheta} implies that $A_\theta(r = 0) = 0$
so long as $A_1, A_2 \in L^\infty_{\mathrm{loc}}$. Moreover, in the Coulomb gauge, 
$A_1$ and $A_2$ exhibit $1/|x|$ decay at infinity and so from \eqref{ArAtheta->A1A2} 
we expect an $L^\infty$ bound for $A_\theta$ but not decay.
The right hand side of \eqref{Atheta} is bounded in absolute value by a constant times $\| \phi \|_{L^2_x}^2$,
which proves \eqref{Athet1}.

For the second inequality, we get using \eqref{Atheta} and Cauchy-Schwarz that
\[
| A_\theta |
\lesssim
\left( \int_0^\infty |\phi|^4 s ds \right)^{1/2} r
\]
Therefore
\[
| \frac1r A_\theta |
\lesssim
\left( \int_0^\infty |\phi|^4 s ds \right)^{1/2}
\]

Finally, to prove \eqref{Athet3}, we use \eqref{direct2} with $n = 1$, first writing
\[
\frac{|A_\theta|}{r^2}
=
\frac12 \cdot \frac{1}{r^2} \int_0^r |\phi|^2 s ds 
=
\frac12 r^{-2} [r^{-1} \bar{\partial}_r]^{-1} |\phi|^2
\]
Therefore
\[
\| \frac{A_\theta}{r^2} \|_{L^p_x} 
\lesssim 
\| r^{-2} [r^{-1} \bar{\partial}_r]^{-1} |\phi|^2 \|_{L^p_x}
\lesssim  \| |\phi|^2 \|_{L^p_x} = \| \phi \|_{L^{2p}_x}^2
\]
\end{proof}

\begin{lem}[Bounds on $A_0$] \label{lem:A0}
Write $A_0 = A_0^{(1)} + A_0^{(2)}$, where
\[
A_0^{(1)} := - \int_r^\infty \frac{A_\theta}{s} |\phi|^2 ds,
\quad \quad
A_0^{(2)} := - \int_r^\infty \frac{m}{s} |\phi|^2 ds
\]
Then
\begin{equation} \label{A01}
\begin{split}
\| A_0^{(1)} \|_{L^1_t L^\infty_x} 
&\lesssim \| \phi \|_{L^4_{t,x}}^4
\\
\| A_0^{(1)} \|_{L^2_{t, x}}
&\lesssim \| \phi \|_{L^\infty_t L^2_x}^2
\| \phi \|_{L^4_{t,x}}^2
, \quad 1 \leq p < \infty
\end{split}
\end{equation}
and
\begin{equation} \label{A02}
\| A_0^{(2)} \|_{L^2_x} \lesssim |m| \| \phi \|_{L^4_x}^2
\end{equation}
\end{lem}
\begin{proof}
The behavior of $A_0$ is independent of the coordinate system.
In particular, it is natural to assume that it decays to zero at infinity as shown in \cite{BeBoSa95}, where
certain $L^p$ bounds are also established under the assumption of sufficient regularity. 
This motivates integrating the $F_{r0}$ curvature condition in \eqref{equiCSS} from infinity,
which is justified below.

To establish the first inequality of \eqref{A01}, rewrite $A_0^{(1)}(r)$ as
\[
A_0^{(1)}(r) = -\int_r^\infty \frac{A_\theta}{s^2} |\phi|^2 s ds
\]
Then, bounding $A_\theta(s) / s^2$ in $L^2$ using \eqref{Athet3}
and putting each $\phi$ in $L^4$, we obtain
\[
|A_0^{(1)}(r)| \leq \| \frac{A_\theta}{s^2} \|_{L^2_x} \| \phi \|_{L^4}^2 \lesssim \| \phi \|_{L^4_x}^4
\]
The bound is independent of $r$, and integrating in time yields
\[
\| A_0^{(1)} \|_{L^1_t L^\infty_x} \lesssim \| \phi \|_{L^4_{t,x}}^4
\]
The second inequality of \eqref{A01} follows from \eqref{Athet1} and \eqref{direct1} with $p = 2$.

To establish \eqref{A02}, we use \eqref{direct1} with $f = m |\phi|^2$ and $p = 2$.
\end{proof}

\begin{lem}[Quadratic bounds]
We have
\begin{align}
\| \frac{1}{r^2} A_\theta^2 \|_{L^1_t L^\infty_x}
&\lesssim
\| \phi \|_{L^4_{t,x}}^4
\\
\| \frac{1}{r^2} A_\theta^2 \|_{L^2_{t,x}}
&\lesssim
\| \phi \|_{L^\infty_t L^2_x}^2 \| \phi \|_{L^4_{t,x}}^2
\end{align}
\end{lem}
\begin{proof}
The first bound follows from \eqref{Athet2} and Cauchy-Schwarz.
The second is a consequence of \eqref{Athet1} and \eqref{Athet3} with $p = 2$.
\end{proof}
Let
\begin{equation} \label{Nphi}
\Lambda(\phi) = 2 m \frac{A_\theta}{r^2} \phi + A_0 \phi + \frac{1}{r^2} A_\theta^2 \phi - g |\phi|^2 \phi
\end{equation}
denote the nonlinearity of the evolution equation of \eqref{equiCSS}.

\begin{rem}
The bounds established in the preceding lemmas are very flexible 
and allow us to control all pieces of the nonlinearity $\Lambda(\phi)$ in $L^{4/3}_{t,x}$ and
some pieces of it in $L^1_t L^2_x$. 
\end{rem}

\begin{lem}[Strichartz estimates] \label{lem:Strichartz}
Let $(i \partial_t + \Delta) u = f$ on a time interval $I$ with $t _0 \in I$ 
and $u(t_0) = u_0$. 
Call a pair $(q, r)$ of exponents admissible if $2 \leq q, r \leq \infty$,
$\frac1q + \frac{1}{r} = \frac12$ and $(q, r) \neq (2, \infty)$.
Let $(q, r)$ and $(\tilde{q}, \tilde{r})$ be admissible pairs of exponents.
Then
\[
\| u \|_{L^\infty_t L^2_x (I \times \R^2)}
+
\| u \|_{L^q_t L^r_x(I \times \R^2)}
\lesssim
\| u_0 \|_{L^2_x(\R^2)}
+
\| f \|_{L^{\tilde{q}^\prime}_t L^{\tilde{r}^\prime}_x(I \times \R^2)}
\]
where the prime indicates the dual exponent, i.e., $\frac{1}{q^\prime} :=
1 - \frac{1}{q}$.
\end{lem}
These estimates are established in \cite{Ya87, GiVe92}.
The only admissible pair that we use in this section is $(q, r) = (4, 4)$.
In the usual way, one may intersect Strichartz spaces. Their dual is then a sum-type space;
we use this property in \S \ref{sec:reg}.
In that section we also use the endpoint estimate, proved in \cite{St01, Ta00}:
\begin{lem}[Endpoint Strichartz esimate] \label{lem:endStrichartz}
Let $(i \partial_t + \Delta) u = f$ on a time interval $I$ with $t _0 \in I$ 
and $u(t_0) = u_0$, and suppose that $m \in \Z$ and $u, f \in L^2_m(\R^2)$.
Let $(q, r)$ be an admissible pair
of exponents. Then
\[
\| u \|_{L^2_t L^\infty_x(I \times \R^2)} \lesssim
\| u_0 \|_{L^2_x(\R^2)} +
\| f \|_{L^{q^\prime}_t L^{r^\prime}_x (I \times \R^2)}
\]
\end{lem}
Though the endpoint estimate was established for radial functions, the
proof may be adapted to equivariant functions in a straightforward way
by noting properties of Bessel functions 
(see, for instance, Remark \ref{rem:radial} for related comments).

\begin{lem}[Control of the nonlinearity] \label{lem:nonlinearity}
We have
\begin{equation} \label{Nbound}
\| \Lambda(\phi) \|_{L^{\frac43}} \lesssim  \| \phi \|_{L^4}^3
\end{equation}
and
\begin{equation} \label{Ndiffbound}
\| \Lambda(\phi) - \Lambda(\tilde{\phi}) \|_{L^{\frac43}} 
\lesssim \| \phi - \tilde{\phi} \|_{L^4}
( \| \phi \|_{L^4}^2 + \| \tilde{\phi} \|_{L^4}^2 )
\end{equation}
\end{lem}
\begin{proof}
The proof is an easy consequence of Strichartz estimates,
charge conservation, and the previous lemmas.
In particular, we have
\[
\begin{split}
\| \frac{2}{r^2} m A_\theta \phi \|_{L^{\frac43}_{t,x}}
&\lesssim
|m| \| \frac{1}{r^2} A_\theta \|_{L^2_{t,x}} \| \phi \|_{L^4_{t,x}}
\lesssim
|m| \| \phi \|_{L^4_{t,x}}^3
\\
\| A_0 \phi \|_{L^{\frac43}_{t,x}} 
&\lesssim \| A_0 \|_{L^2_{t,x}} \| \phi \|_{L^4_{t,x}}
\lesssim (|m| + \| \phi \|_{L^\infty_t L^2_x}^2) \| \phi \|_{L^4_{t,x}}^3
\\
\| \frac{1}{r^2} A_\theta^2 \phi \|_{L^{\frac43}_{t,x}} 
&\lesssim \| \frac{1}{r^2} A_\theta^2 \|_{L^2_{t,x}}
\| \phi \|_{L^4_{t,x}}
\lesssim
\| \phi \|_{L^\infty_t L^2_x}^2 \| \phi \|_{L^4_{t,x}}^3
\\
\|g |\phi|^2 \phi \|_{L^{\frac43}_{t,x}} 
&\leq |g| \| \phi\|_{L^4_{t,x}}^3
\end{split}
\]
which establishes \eqref{Nbound}.

The second inequality is easy to show for the nonlinear term $g |\phi|^2 \phi$ by using
the observation
\begin{equation} \label{nonobs}
\left\lvert |\phi|^2 \phi - | \tilde{\phi} |^2 \tilde{\phi} \right\rvert
\lesssim
\left( |\phi|^2 + |\tilde{\phi}|^2 \right) \lvert \phi - \tilde{\phi} \rvert
\end{equation}
To see that others are similar, note that bounds \eqref{Athet1}--\eqref{Athet3}
for $A_\theta = A_\theta(\phi)$ are linear in $|\phi|^2$.
This is also true of the bound for $A_0^{(2)}$ in Lemma \ref{lem:A0}.
Applying further decompositions similar to \eqref{nonobs} 
allows one to handle the higher-order terms
$\frac{1}{r^2} A_\theta^2$ and $A_0^{(1)}$.
\end{proof}

In our analysis, the $L^4_{t,x}$ norm plays the role of a scattering norm.
If $\phi: I \times \R^2 \to \C$ is a solution of \eqref{equiCSS} on an open time interval $I$,
then we say that $\phi$ blows up forward in time if 
$\| \phi \|_{L^4_{t,x}((I \cap [t, \infty)) \times \R^2)} = \infty$ for all $t \in I$.
Similarly, we say that $\phi$ blows up backward in time if
$\| \phi \|_{L^4_{t,x}((I \cap (-\infty, t]) \times \R^2)} = \infty$ for all $t \in I$.

Let $\phi_+ \in L^2$. We say that a solution $\phi : I \times \R^2 \to \C$ scatters
forward in time to $\phi_+$ if and only if $\sup I = + \infty$ and
$\lim_{t \to \infty} \| \phi(t) - e^{i t \Delta} \phi_+ \|_{L^2} = 0$.
Similarly, we say that a solution $\phi : I \times \R^2 \to \C$ scatters
backward in time to $\phi_- \in L^2$ if and only if $\inf I = - \infty$ and
$\lim_{t \to - \infty} \| \phi(t) - e^{i t \Delta} \phi_- \|_{L^2} = 0$.

\begin{thm}[Cauchy theory] \label{thm:Cauchy}
Let $m \in \Z$, $\phi_0 \in L^2_m(\R^2)$, and $t_0 \in \R$.
There exists a unique maximal lifespan solution $\phi : I \times \R^2 \to \C$, $\phi \in L^\infty_I L^2_m$,
with $t_0 \in I$, the maximal time interval, $\phi(t_0) = \phi_0$, and with the following additional properties:
\begin{enumerate}
\item
(Local existence) $I$ is open.
\item
(Scattering) 
If $\phi$ does not blow up forward in time, then $\sup I = + \infty$
and $\phi$ scatters forward in time to $e^{i t \Delta} \phi_+$ for some $\phi_+ \in L^2_m$.
If $\phi$ does not blow up backward in time, then $\inf I = - \infty$ and $\phi$
scatters backward in time.
\item
(Small data scattering) 
There exists $\varepsilon > 0$ such that if $\| \phi_0 \|_{L^2} \leq \varepsilon$, then
$\| \phi \|_{L^4_{t,x}} \lesssim \| \phi_0 \|_{L^2_x}$.
In particular, $I = \R$ and the solution scatters both forward and backward in time.
\item
(Uniformly continuous dependence)
For every $A > 0$ and $\varepsilon > 0$ there is a $\delta > 0$ such that if $\phi$
is an $m$-equivariant solution satisfying $\| \phi \|_{L^4_{t,x}(J \times \R^2)} \leq A$ with $t_0 \in J$
and if $\tilde{\phi}_0 \in L^2_m$ satisfies $\| \phi_0 - \tilde{\phi}_0 \|_{L^2_x} \leq \delta$, 
then there exists an $m$-equivariant solution $\tilde{\phi}$
such that $\| \phi - \tilde{\phi} \|_{L^4_{t,x}(J \times \R^2)} \leq \varepsilon$ and
$\| \phi(t) - \tilde{\phi}(t) \|_{L^2} \leq \varepsilon$ for all $t \in J$.
\item
(Stability)
For every $A > 0$ and $\varepsilon > 0$ there exists $\delta > 0$ such that if
$\| \phi \|_{L^4_{t,x}(J \times \R^2)} \leq A$, $\phi$ is $m$-equivariant and approximates \eqref{equiCSS}
in that $\| (i \partial_t + \Delta) \phi - \Lambda(\phi) \|_{L^{4/3}_{t,x}(J \times \R^2)} \leq \delta$,
$t_0 \in J$, and $\tilde{\phi}_0 \in L^2_m$ satisfies
$\| e^{i (t - t_0)\Delta} (\phi(t_0) - \tilde{\phi}_0) \|_{L^4_{t,x}(J \times \R^2)} \leq \delta$, then there exists an $m$-equivariant solution $\tilde{\phi}$
with $\tilde{\phi}(t_0) = \tilde{\phi}_0$ and $\| \phi - \tilde{\phi} \|_{L^4_{t,x}(J \times \R^2)} \leq \varepsilon$.
\end{enumerate}
\end{thm}
\begin{proof}
The local existence statement follows from \eqref{Nbound} and a standard iteration argument.
The scattering claim (2)
follows from \eqref{Nbound} and from linearizing near the asymptotic states.
The remaining claims follow from \eqref{Ndiffbound} by standard arguments.
\end{proof}

\section{Frequency localization} \label{sec:freq}

The purpose of this section is to relate Littlewood-Paley frequency-localizations
of terms of $\Lambda(\phi)$, defined in \eqref{Nphi}, to frequency localizations of $\phi$.
This is done in a way that respects the $L^p$ estimates established in the previous section.

We introduce Littlewood-Paley multipliers in the usual way. In particular, let
$\psi : \R^+ \to [0, 1]$, $\psi \in C^\infty$, equal one on $[0, 1]$ and zero on $[2, \infty)$.
For each $\lambda > 0$, define
\[
\mathcal{F}(P_{\leq \lambda} f)(\xi) := \psi(|\xi| \lambda^{-1}) \hat{f}(\xi), \quad
\mathcal{F}(P_{> \lambda} f)(\xi) := \left(1 - \psi(|\xi| \lambda^{-1}) \right) \hat{f}(\xi)
\]
\[
\widehat{P_\lambda f}(\xi) := \left( \psi(|\xi| \lambda^{-1}) - \psi(2 |\xi| \lambda^{-1}) \right) \hat{f}(\xi)
\]
We similarly define $P_{\leq \lambda}$ and $P_{\geq \lambda}$. Also, for $\lambda > \mu > 0$, set
\[
P_{\mu < \cdot \leq \lambda} := P_{\leq \lambda} - P_{\leq \mu}
\]
The standard $L^p$ Bernstein estimates hold for these multipliers, e.g., see
\cite[Lemma 2.1]{KiTaVi09}.

We record for reference the useful relation
\begin{equation} \label{xgradx}
\begin{split}
\mathcal{F}(r \partial_r f) =
\mathcal{F}(x \cdot \nabla f) 
&= \mathcal{F}(x^j \partial_j f)
\\
&= i \partial_{\xi_j} (i \xi_j \hat{f})
= - 2 \hat{f} - \xi \cdot \nabla_{\xi} \hat{f} = - 2 \hat{f} - \rho \partial_\rho \hat{f}
= - \rho^{-1} \partial_\rho ( \rho^2 \hat{f} )
\end{split}
\end{equation}
which is valid when the dimension of the underlying space is $2$. 
Here and throughout we set $\rho := |\xi|$.
We also set
\begin{equation} \label{f(r)}
f(r) := - \frac12 |\phi|^2
\end{equation}
for short and note the following equalities, which
follow from \eqref{Atheta}:
\begin{equation} \label{Athetachain}
\frac1r \partial_r A_\theta = 
\left(\frac1r +  \partial_r \right) \left( \frac1r A_\theta \right) =
(2 + r \partial_r) \left(\frac{1}{r^2} A_\theta\right) = f(r)
\end{equation}

\begin{lem}[Fourier transforms of $A_\theta$, $r^{-2} A_\theta$] \label{lem:FourierA}
Let $f$ be given by \eqref{f(r)}. Then
\begin{equation} \label{Athetahat}
\hat{A}_\theta = \rho^{-1} \partial_\rho \hat{f}
\end{equation}
and
\begin{equation} \label{Abighat}
\mathcal{F}(r^{-2} A_\theta) = - [\rho \partial_\rho]^{-1} \hat{f}
\end{equation}
\end{lem}
\begin{proof}
We invoke \eqref{A1A2->ArAtheta} to get
\[
\hat{A}_\theta(\xi) = - i \partial_{\xi_2} \hat{A}_1 + i \partial_{\xi_1} \hat{A}_2
\]
where we interpret the derivatives in the sense of distributions.
Upon expansion we write
\[
\hat{A}_\theta(\xi)
=
- \frac12 \partial_{\xi_2} \left( \frac{\xi_2}{|\xi|^2} \mathcal{F}(|\phi|^2) \right)
-
\frac12 \partial_{\xi_1} \left( \frac{\xi_1}{|\xi|^2} \mathcal{F}(|\phi|^2) \right)
\]
This simplifies to
\[
\hat{A}_\theta(\xi) = -\frac12 \frac{\xi_j}{|\xi|^2} \partial_{\xi_j} \mathcal{F}(|\phi|^2)
\]
so that
\[
\hat{A}_\theta(\rho) = - \frac{1}{2 \rho} \partial_\rho \mathcal{F}(|\phi|^2)
\]
which establishes \eqref{Athetahat};
alternatively, one may multiply \eqref{Athetachain} by $r^2$ and use \eqref{xgradx}.
To show \eqref{Abighat}, let
\[
F(r) := \frac{A_\theta}{r^2} = \frac{1}{r^2} \int_0^r f(s) s ds
\]
where the equality follows from \eqref{Atheta}.
This function is differentiable a.e.~and satisfies
\[
(2 + r \partial_r) F(r) = f(r)
\]
as noted in \eqref{Athetachain}.
Taking Fourier transforms and using \eqref{xgradx}, we obtain
\[
\hat{f} = 2 \widehat{F} + \nabla_{\xi} \cdot \xi \widehat{F}
=
- \rho \partial_\rho \widehat{F}
\]
Because $\phi \in L^4_{t,x}$, it follows that $\phi \in L^4_x$ for a.e.~$t$
and hence $f \in L^2_x$ for a.e.~$t$.
Therefore, writing $\widehat{F} = -[\rho \partial_\rho]^{-1} \hat{f}$,
we may invoke \eqref{direct1} for a.e.~$t$ with $p = 2$
and so
conclude that the Fourier transform of $A_\theta / r^2$ has the desired localization properties.
\end{proof}

\begin{lem}[Fourier transform of $A_0^{(1)}$]
Let $G(r) = r^{-2} A_\theta |\phi|^2$. Then
\[
\hat{A}_0^{(1)} = \rho^{-1} \partial_\rho \widehat{G}
\]
\end{lem}
\begin{proof}
Note that
\[
r \partial_r A_0^{(1)} = A_\theta |\phi|^2
\]
a.e.~so that in particular
\begin{equation} \label{LapA01p1}
\frac1r \partial_r A_0^{(1)} = \frac{A_\theta}{r^2} |\phi|^2
\end{equation}
From this we also obtain
\begin{equation} \label{LapA01p2}
\partial_r^2 A_0^{(1)} = - \frac{A_\theta}{r^2} |\phi|^2 + \frac{1}{r} \partial_r (A_\theta |\phi|^2)
\end{equation}
which is valid in the sense of distributions.
Combining \eqref{LapA01p1} and \eqref{LapA01p2} and using the fact that $A_0^{(1)}$
is radial, we conclude
\[
A_0^{(1)} = \Delta^{-1} \left[ \frac1r \partial_r (r^2 G) \right] = \Delta^{-1} (2 + r \partial_r) G(r)
\]
Invoking \eqref{xgradx} (with the roles of $r$ and $\rho$ reversed), we get
\[
\hat{A}_0^{(1)} = -\frac{1}{\rho^2} \left(- \rho \partial_\rho \widehat{G} \right) =
 \frac{1}{\rho} \partial_\rho \widehat{G}
\]
\end{proof}

\begin{lem}[Fourier transform of $A_0^{(2)}$]
The following holds:
\[
P_N \left([r \partial_r]^{-1} |P_{< N} \phi|^2\right) = 0
\]
\end{lem}
\begin{proof}
The term $A_0^{(2)}$ is nonzero only in the nonradial equivariant case.
In particular, we have
\[
r \partial_r A_0^{(2)} = m |\phi|^2
\]
a.e.~from the representation given in Lemma \ref{lem:A0} and
\begin{equation} \label{A02Rad}
\Delta A_0^{(2)} = m \frac1r \partial_r |\phi|^2
\end{equation}
in the sense of distributions.
The Cartesian coordinate representation
\begin{equation} \label{A02Cart}
\Delta A_0^{(2)} = \Im(Q_{12}(\bar{\phi}, \phi)),
\end{equation}
however, is more convenient for our purposes here.
In particular, we see immediately that
\begin{equation} \label{A02Horse}
P_N \Im(Q_{12}(P_{< N} \bar{\phi}, P_{< N} \phi)) = 0
\end{equation}
so that any contribution to $P_N A_0^{(2)}$ must come from input $\phi$-frequencies
of at least frequency $N$.
\end{proof}
\begin{rem}
Together \eqref{A02Cart} and \eqref{A02Horse} suggest splitting each $\phi$ input in the right-hand side
of \eqref{A02Cart} into a sum of Littlewood-Paley frequency localizations.
As $Q_{12}(\cdot, \cdot)$ is linear in each argument separately, there are some cross terms
to handle, e.g., terms of the form $\Im(Q_{12}(\overline{P_J \phi}, P_K \phi))$ with ranges
$J$ and $K$ not equal.
Whereas the Cartesian representation is well-suited for revealing the frequency localization,
it is the radial representation \eqref{A02Rad} that is used in Lemma \ref{lem:A0} in
proving the $L^2$ estimate of that lemma, which does not hold for arbitrary (non-equivariant) $L^2$ data.
Therefore, in order to take advantage of this frequency decomposition,
we need to ensure that we can apply the $L^2$ estimate to terms of the form
$\Im(Q_{12}(\overline{P_J \phi}, P_K \phi))$.
Note that if $\phi$ is $m$-equivariant, then so are $P_J \phi$, $P_K \phi$, so that
both inputs of $\Im(Q_{12}(\cdot, \cdot))$ are $m$-equivariant.
In particular, if both $\phi$ and $\psi$ are $m$-equivariant, then
\[
\Im(Q_{12}(\bar{\phi}, \psi)) = m \frac{1}{r} \partial_r \Re(\bar{\phi} \psi),
\]
and so we may use \eqref{direct1} as in the proof of \eqref{A02} of Lemma \ref{lem:A0}.
\end{rem}

\section{Extra regularity} \label{sec:reg}

\subsection{The self-similar case}

Our goal in this section is to show that
self-similar minimal solutions enjoy extra regularity.
\begin{thm}\label{ExtraRss} 
Let $\phi$ be a self-similar critical $m$-equivariant solution of \eqref{equiCSS}, 
almost periodic modulo scaling, 
with lifespan $I = (0, +\infty)$, $N(t) = t^{-1/2}$ for $t \in I$. Then, for each $s \geq 0$, 
$\phi \in L^\infty_t H^s_m (\R \times \R^2)$.
\end{thm}
We adopt the basic setup used in \cite{KiTaVi09}
and introduce the quantities
\begin{align*}
\mathcal{M}(A) & =\sup_T\|P_{>AT^{-\frac12}} \phi(T)\|_{L^2_x(\R^2)} \\
\mathcal{S}(A) & =\sup_T\|P_{>AT^{-\frac12}} \phi(t,x)\|_{L^4_{t,x}([T, 2T]\times \R^2)}  \\
\mathcal{N}(A) & =\sup_T\|P_{>AT^{-\frac12}} \Lambda(\phi)(t,x)
\|_{L^{\frac43}_{t,x}([T,2T]\times \R^2)+L^1_t[T,2T]L^2_x(\R^2)}  \\
\widetilde{\mathcal{N}}(A)& =\sup_T \|P_{>AT^{-\frac12}} \Lambda(\phi)(t,x)\|_{L^{\frac43}_{t,x}([T,2T]\times\R^2)}  
\end{align*}
For our definition of Littlewood-Paley multipliers, see \S \ref{sec:freq}. The nonlinearity $\Lambda(\phi)$
is defined in \eqref{Nphi}.
Whereas $\widetilde{\mathcal{N}}(A)$ is used in \cite{KiTaVi09, GuKo11}
to prove extra regularity for self-similar solutions, 
we use the slightly weakened norm $\mathcal{N}(A)$. 
This is especially helpful when $A_\theta$ has high frequency inputs, as shown in 
Lemma \ref{NLEstimate}. 
 
To prove Theorem \ref{ExtraRss}, we will show that
\begin{equation} \label{ExtraRsscon}
\mathcal{M}(A)=\sup_T\|P_{>AT^{-\frac12}} \phi(T)\|_{L^2_x} <_{\phi ,A}A^{-s}
\end{equation}
for any $s>0$.

\subsubsection{Bounds}

Mass conservation gives 
\[
\ma \lesssim_u 1
\]
and Strichartz estimates imply
\begin{equation}\label{saBound}
\sa \lesssim_\phi \ma +\na
\end{equation}
The spacetime bound proved in \cite[Lemma 3.9]{KiTaVi09} establishes
\[
\sa \lesssim_\phi 1
\]
and this spacetime bound together with 
Lemma \ref{lem:nonlinearity} 
implies
\[
\na
\lesssim \widetilde{\mathcal{N}}(A)
\lesssim\| \Lambda(\phi)\|_{L^{\frac43}{([T,2T] \times \R^2)}}
\lesssim_\phi 1
\]  
The Strichartz estimate together with the above inequalities implies 
\[
\|\phi \|_{L^2_tL^\infty_x ([T, 2T]\times \R^2)} \lesssim_\phi 1
\]
In the following lemma we collect some estimates that we will later employ.
\begin{lem} \label{Holder}
Suppose $\frac{1}{p} =\frac{1}{p_1} +\frac{1}{p_2} +\frac{1}{p_3}$ and $\frac{1}{p_2} +\frac{1}{p_3} >0$.
Then the following
nonlocal H\"{o}lder estimate holds
\begin{equation}\label{Holder1}
\|q_1\int_r^\infty q_2q_3 \frac{d\rho}{\rho}\|_{p}
\lesssim 
\|q_1\|_{p_1}\|q_2\|_{p_2}\|q_3\|_{p_3}
\end{equation}
Additionally,
\begin{equation}\label{Holder2}
\|\frac{A_\theta}{r^2}\phi\|_{p}
\lesssim 
\|\phi\|_{p_1} \|\phi\|_{p_2} \|\phi\|_{p_3}
\end{equation}
for $\frac{1}{p}=\frac{1}{p_1}+\frac{1}{p_2} +\frac{1}{p_3}$ with $1<p_i<\infty$.
 
Also true for equivariant functions $f$ is the Strichartz estimate
\begin{equation}\label{Radial1}
\|P_N e^{it\Delta} f\|_{L^q}
\lesssim 
N^{1-\frac{4}{q}}\|f\|_{L^2_x}, 
\quad q \geq \frac{10}{3}
\end{equation}
from which easily follows the inhomogeneous estimate
\begin{equation}\label{Radial2}
\|P_N u\|_{L^q}
\lesssim 
N^{1-\frac{4}{q}}(\|f\|_{L^2_x} +\|(i\partial_t +\Delta) u\|_{L^\frac{4}{3}+L^1_tL^2_x}), 
\quad q \geq \frac{10}{3}
\end{equation}
\end{lem}
The nonlocal H\"older estimate follows from elementary inequalities, see \cite[\S 3]{GuKo11}.
Shao \cite{Sh09} proved \eqref{Radial1} for the range $q > \frac{10}{3}$, 
and the endpoint $q = \frac{10}{3}$ was established
by Guo and Wang in \cite{GuWa10}.

\begin{rem} \label{rem:radial}
There is enough slack in our argument for nonendpoint estimates to suffice.
However, when the endpoint estimate is used, the exponents are particularly simple,
and so we use this estimate for convenience. 
Note that both \cite{Sh09} and \cite{GuWa10} prove results for radial functions. 
There they use the fact that the Fourier transform of a radial function may be expressed in terms of a
Hankel transform with kernel $J_k$, a Bessel function of the first kind. 
When the underlying space is two-dimensional, $k = 0$.
In the $m$-equivariant 2-d setting, the Bessel function required is $J_m$, which enjoys the same
asymptotics at infinity as does $J_0$, but is better behaved near the origin. 
These properties are sufficient for extending the proofs of \cite{Sh09, GuWa10} to this setting.
\end{rem}

We now come to the first main estimate.
\begin{lem}\label{NLEstimate}
Given $A$ large enough, we have 
\[
\na \lesssim_\phi
\mathcal{S}(\frac{A}{10})\mathcal{M}(\sqrt{A})  + 
A^{-\frac{1}{10}}[\mathcal{M}(\frac{A}{10})+\mathcal{N}(\frac{A}{10})]
\]
 \end{lem}
 \begin{proof} 
 We proceed as in \cite{KiTaVi09}. It suffices to prove 
 \[
 \|P_{>AT^{-\frac12}} \mathcal{N}(\phi)(t,x)\|_{L^{\frac43}_{t,x} +L^1_tL^2_x[T,2T]} 
 \lesssim
 \mathcal{S}(\frac{A}{10})\mathcal{M}(\sqrt{A})  + 
A^{-\frac{1}{10}}[\mathcal{M}(\frac{A}{10})+\mathcal{N}(\frac{A}{10})]
 \] 
uniformly in $T$.
To do this, we decompose $\phi$ into high, intermediate and low frequency pieces, i.e.,
\[
\phi = \phi_{hi} +\phi_{med}+ \phi_{low}
\]
where 
\[
\phi_{hi}=\phi_{> \frac{1}{10}AT^{-\frac12}}, \quad
\phi_{med}= \phi_{\sqrt{A}T^{-\frac12}\leq \cdot \leq \frac{1}{10}AT^{-\frac12}}, \quad
\phi_{low}=\phi_{\leq \sqrt{A}T^{-\frac12}}
\]
Because of the frequency localization lemmas of \S \ref{sec:freq},
we see that having nontrivial $P_{>AT^{-\frac12}} \mathcal{N}(\phi)(t,x)$ 
implies that the nonlinearity must have at least one high frequency input  $\phi_{hi}$.
  
As in \cite{KiTaVi09}, we split into cases according to whether we have one intermediate input or all low inputs.
 
It is convenient at this stage to split up the nonlinearity into ``cubic" and ``quintic" terms, as follows
\begin{equation} \label{Nsplit}
\Lambda_3 := 2m \frac{A_\theta}{r^2} \phi + A_0^{(2)} \phi - g |\phi|^2 \phi, \quad
\Lambda_{5,1} := \frac{A_\theta^2}{r^2} \phi, \quad
\Lambda_{5,2} := A_0^{(1)} \phi,
\end{equation}
so that $\Lambda(\phi) = \Lambda_3 + \Lambda_{5,1} + \Lambda_{5,2}$.

\textbf{Case 1: }
If we have at least one intermediate input $ \phi_{med}$ in the nonlinearity, 
then we use the H\"{o}lder estimate \eqref{Holder1}. 
In particular, we use $L^4$ on $\phi_{hi}$ and $L^{\infty}_tL^2_x$ on $\phi_{med}$, 
and so obtain the bound
\[ 
\| \Lambda_3(\phi_{hi}, \phi_{med}, \phi)\|_{L^{\frac43}[T,2T]}
\lesssim 
\mathcal{S}(\frac{A}{10})\mathcal{M}(\sqrt{A})
\]
For the quintic term $\Lambda_{5,1}$, use $L^\infty$ on an $A_\theta$ that does not involve $\phi_{hi}$ 
and then apply H\"older to $\frac{A_\theta}{r^2}\phi$ in the same way that we do for
the cubic terms:
\[
\|\Lambda_{5,1}\|_{L^{\frac43}[T,2T]}\lesssim_\phi \mathcal{S}(\frac{A}{10})\mathcal{M}(\sqrt{A})  
\]
We can control the quintic term $\Lambda_{5,2}$
in $L^{\frac43}$ using $L^\infty$ on $A_\theta$ and H\"older on the other terms
provided that $A_\theta$ does not have a high frequency input. 
If $A_\theta$ does have a high frequency input, then we estimate $\Lambda_{5,2}$ in $L^1_tL^2_x$:
\begin{align*}
\|\Lambda_{5,2}\|_{L^1_tL^2_x}& 
\lesssim 
\|\int_r^\infty \frac{  A_\theta}{s^2} |\phi|^2 sds\|_{L^1_tL^\infty_x}\|\phi\|_{L^\infty_tL^2_x[T,2T]}
\\ 
&\lesssim \|\frac{A_\theta}{s^2} |\phi|^2\|_{L^1_{t,x}[T,2T]}\|\phi\|_{L^\infty_tL^2_x}
\\
& \lesssim \mathcal{S}(\frac{A}{10})\mathcal{M}(\sqrt{A})
\end{align*}
Here we used the H\"{o}lder estimate \eqref{Holder2}, putting the high frequency terms in $L^4$, 
the medium frequency ones in $L^\infty L^2$, and the rest in $L^4$ and $L^2L^\infty.$ 
 
Altogether, we conclude
$\|\Lambda_{5,2}\|_{L^{\frac43}+L^1L^2}\lesssim_{\phi }\mathcal{S}(\frac{A}{10})\mathcal{M}(\sqrt{A})$.

\textbf{Case 2:} 
For the case where one input is at high frequency and the rest are at low frequency,
we adopt the idea of using the Strichartz estimates \eqref{Radial1}, \eqref{Radial2}, as found in
\cite[\S 3.3]{GuKo11}.
 
For $\Lambda_3$, we use, as in Case 1, $L^{\frac{10}{3}}$ on $\phi_{hi}$ and $L^5$ on one of $\phi_{low}$:
\[ 
\|\Lambda_3(\phi)\|_{L^\frac{4}{3}{[T,2T]}}
\lesssim 
\|\phi_{hi}\|_{L^{\frac{10}{3}}{[T,2T] }  }\|\phi_{low}\|_{L^{5}{[T,2T]}}\|\phi\|_{L^{4}{[T,2T]}}
\]
Using Bernstein and the inhomogeneous Strichartz estimate (\ref{Radial2}), we get 
\[
\begin{split}
\|P_{<M}\phi\|_{L^{5}_{[T,2T]\times \mathbb{R}^2}}
&\lesssim 
M^{\frac{1}{5}}\left( \|P_{<M}\phi(T)\|_{L^2} + \|P_{<M}\mathcal{N}(\phi)\|_{L^1L^2+L^{\frac43}[T,2T]} \right)
\\
\|P_{> N}\phi\|_{L^{\frac{10}{3}}_{[T,2T]\times \mathbb{R}^2}}
&\lesssim 
N^{-\frac{1}{5}}\left( \|P_{> N}\phi(T)\|_{L^2} +\|P_{>N}\mathcal{N}(\phi)\|_{L^1L^2+L^{\frac43}[T,2T]} \right)
\end{split}
\] 
Taking $N=\frac{1}{10}AT^{-\frac12}$ and  $M=\sqrt{A}T^{-\frac12}$, 
we obtain
\[ 
\|\Lambda_3(\phi)\|_{L^\frac{4}{3}_{[T,2T]\times \mathbb{R}^2}}
\lesssim 
(A)^{-\frac{ 1}{10}}[\mathcal{M}(\frac{1}{10}A)+\mathcal{N}(\frac{A}{10})] 
\]
The quintic pieces of $\Lambda_{5,1}$ and $\Lambda_{5,2}$ with $A_\theta$ not involving $\phi_{hi}$
we handle as in Case 1. In particular, we use $L^\infty$ on $A_\theta$,
then apply H\"older to obtain $\|\phi_{hi}\|_{L^{\frac{10}{3}}  }$ and $\|\phi_{low}\|_{L^{5}}$,
and then apply H\"older once more to get the $A^{-\frac{1}{10}}$ decay factor. 
 
The quintic term $\Lambda_{5,2}$ with $A_\theta$ involving $\phi_{hi}$ we bound in $L^1L^2$ 
as in Case 1:
\begin{align*}
\|\Lambda_{5,2}\|_{L^1_tL^2_x}& 
\lesssim \|\phi_{hi}\|_{L^{\frac{10}{3}}}\|\phi\|_{L^5}\|\phi\|^2_{L^{4}} \|\phi\|_{L^\infty L^2 }  
\lesssim A^{-\frac{1}{10}}[\mathcal{M}(\frac{1}{10}A)+\mathcal{N}(\frac{A}{10})]  
\end{align*}
\end{proof}

\begin{lem} 
We have
\[
\lim_{A\rightarrow \infty} \ma, \sa,\na =0
\]
\end{lem}
\begin{proof} 
The proof follows that in \cite{KiTaVi09}, 
using the definition of almost periodicity, self-similarity, Lemma \ref{NLEstimate}, and (\ref{saBound}).
\end{proof}

Given the nonlinear estimate established in Lemma \ref{NLEstimate},
the following $\varepsilon$-regularity result follows using exactly the same arguments
employed in \cite[Prop.~5.5]{KiTaVi09}.
\begin{lem}\label{Decay}
For all $A > 0$,
\[
S(A)\lesssim \eta S(\frac{A}{20})  +A^{-\frac{1}{40}}
\]
and
\[
\ma+\sa+\na\lesssim A^{-\frac{1}{40}}
\]
\end{lem}
Finally, adapting the induction argument, we conclude higher regularity.
\begin{thm} 
For all $A > 0$ and $s > 0$,
\[
\ma \lesssim A^{-s}
\]
\end{thm}

\subsection{The global critical case}

The Fourier transform of an $m$-equivariant function $f(r, \theta) = e^{im\theta}u(r) $ 
is given in terms of a Hankel transform of its radial part $u$. We use polar coordinates
$(\rho, \alpha)$ on the Fourier side, obtaining
\[
\mathcal{F}(f)(\rho, \alpha) 
= 
2\pi(-i)^m e^{im\alpha} \int_0^\infty u(r)J_m(r \rho)rdr
\]
The Fourier transform is an involution on equivariant functions, and so one may also
obtain from this an inversion formula.
Next, we split the Bessel function $J_m$ into two Hankel functions, corresponding
to projections onto outgoing and incoming waves. In particular, we have
\[
J_m(|x||\xi|)=\frac{1}{2} H_m^{(1)}(|x||\xi|) + \frac12 H_m^{(2)}(|x||\xi|)
\]
where $H_m^{(1)}$ is the order $m$ Hankel function of the first kind and $H_m^{(2)}$ 
is the order $m$ Hankel function of the second kind.
\begin{defin} 
Let $P^{+}$ denote the projection onto outgoing $m$-equivariant waves
  \begin{align*}[P^{+}f](x)
  & := 
  \frac{1}{4\pi^2}e^{im\theta} \int_{\mathbb{R}^{+}\times\mathbb{R}^{+}}H^{(1)}_m(|x||\xi|)J_m(|\xi||y|)f(|y|) d\xi dy
 \\ 
 & \phantom{:}= 
 \frac{1}{2}f(x) +\frac{i}{2\pi^2}\int_{\mathbb{R}^2}\left|\frac{y}{x}\right|^m\frac{f(x)}{|x|^2-|y|^2}dy \end{align*}
Here the second inequality follows from \cite[\S 6.521.2]{GrRy07} and analytic continuation.

In a similar way, we can define the projection $[P^{-}f](x)$ onto incoming waves
by replacing $H^{(1)}_m$ with $H^{(2)}_m$. 
In particular, $P^{-}f$ is the complex conjugate of $P^{+}f$.
  
We also use the notation $P^{\pm}_N$ for the composition $P^{\pm}P_N$.

As the equivariance class $m$ is clear from context, we omit it from the notation for $P^{\pm}$.
\end{defin}

\begin{lem}[Kernel Estimate] \label{lem:KE}
\begin{enumerate}
\item 
The operator $P^{+}+P^{-}$ acts as the identity on $m$-equivariant functions belonging to $L^2(\mathbb{R}^2)$.
\item 
For $|x|>N^{-1}$ and $t\gtrsim N^{-2}$,
 \[
 |[P^{\pm}_N e^{\mp it\Delta}](x,y)|
 \lesssim
 \left\{
 \begin{array}{ll}
 (|x||y||t|)^{-\frac12}: &|y|-|x|\sim Nt
 \\
 \frac{N^2}{\langle N|x| \rangle^{\frac12} \langle N|y| \rangle^{\frac12}} \langle N^2t + N|x|-N|y|\rangle^{-n}: &otherwise
\end{array}
\right.
\] 
for all $n\geq 0$.
\item 
For $|x|\gtrsim N^{-1}, t\lesssim N^{-2}$. 
 \[
 |[P^{\pm}_N e^{\mp it\Delta}](x,y)|\lesssim  \frac{N^2}{\langle N|x| \rangle^{\frac12} \langle N|y| \rangle^{\frac12}} \langle   N|x|-N|y|\rangle^{-n}  
 \] 
 for all $n\geq 0.$
 \item 
 For $N>0$ and any equivariant function $f\in L^2(\mathbb{R}^2)$,
 \[\|P^{\pm}P_{\geq N}f\|_{L^2(|x|\geq \frac{1}{100N})}\lesssim \|f\|_{L^2(\mathbb{R}^2)}\]
\end{enumerate}
\end{lem}
This result is established in \cite{KiTaVi09, KiViZh08}; see for instance
\cite[Prop.~6.2]{KiTaVi09} and \cite[Lem.~4.1]{KiViZh08}.
The spatial cutoff in (4) of Lemma \ref{lem:KE} is only necessary when $m \neq 0$. 
The operators $P^{\pm}$ that act on radial functions
are bounded on $L^2$. However, their counterparts for $m \neq 0$ are no longer bounded on $L^2$ because
of the worse singularity of $H_m^{(1)}$ (and $H_m^{(2)}$) at the origin. 

With the help of the decay provided by the incoming/outgoing wave decompositions, we can prove
the following lemma.
\begin{thm}\label{ExtraR} 
Let $\phi$ be a global critical $m$-equivariant solution of \eqref{equiCSS}, almost periodic modulo scaling, 
and with $N(t) \lesssim 1$ uniformly in $t \in \R$. Then, for each $s \geq 0$, 
$\phi \in L^\infty_t H^s_m (\R \times \R^2)$.
\end{thm}
It suffices to prove
\[
\mathcal{M}(\lambda):=\|\phi_{\geq \lambda}\|_{L^\infty_tL^2_x(\mathbb{R}\times\mathbb{R}^2)}
\lesssim \lambda^{-s}
\]
By mass conservation,
\[
\|\phi\|_{L^\infty_tL^2_x(\mathbb{R}\times\mathbb{R}^2)}\lesssim_\phi 1
\] 
and so $\mathcal{M}(\lambda)\gtrsim 1$.
From almost periodicity and from the boundedness of $N(t)$, we get 
\[
\lim_{\lambda \rightarrow \infty}\|\phi_{\geq \lambda}\|_{L^\infty_t L^2_x(\mathbb{R}\times\mathbb{R}^2)}=0,
\] 
which means that $\mathcal{M}(\lambda)\rightarrow 0$.
 
As we can see, Theorem~\ref{ExtraR} follows from the following lemma
\begin{lem}[Regularity]
Let $\phi$ be as in Theorem~\ref{ExtraR} and let $\eta > 0$ be a small number.
Then
\[
\mathcal{M}(\lambda)\lesssim \eta\mathcal{M}(\frac{\lambda}{8})
\] 
whenever $\lambda$ is sufficiently large, depending upon $\phi$ and $\eta$.       
\end{lem}
We prove this lemma by showing that 
\begin{equation} \label{sufflem}
 \|\phi_{\geq \lambda}(t_0)\|_{ L^2_x( \mathbb{R}^2)}\lesssim \eta \mathcal{M}(\frac{\lambda}{8})
\end{equation}
for all time $t_0$ and $\lambda$ sufficiently large. 
      
Let us explain the idea briefly: Because we are at a short time interval, almost periodicity and boundedness of $N(t)$ imply that the solution has little mass at high frequency. When we are at a long time interval, we can split into incoming and outgoing waves, which will diminish as it moves away from the origin. 

We carry out the ideas in detail. We can first assume $t_0=0$ by time translation.
Let $\chi_\lambda(x)$ denote the characteristic function of $[\frac{1}{\lambda},\infty)$.
      
For the portion of the frequency localized solution $\phi_{\geq \lambda}$ in the ball $\{ |x| \leq \lambda^{-1} \}$,
we get (see \cite[(5-7)]{KiViZh08}) that
\begin{align}
(1-\chi_\lambda(x))\phi_{\geq \lambda}(0) 
 =& \lim_{T\rightarrow \infty} i\int_0^T  (1-\chi_\lambda(x)) e^{-it\Delta} P_{\geq \lambda}\Lambda(\phi)(t)dt \nonumber
 \\
 =&\;  i\int_0^\delta (1-  \chi_\lambda(x))e^{-it\Delta} P_{\geq \lambda}\Lambda(\phi)(t)dt  \nonumber
 \\      
 &+\lim_{T\rightarrow \infty} \sum_{\mu \geq \lambda} i \int_\delta^T \int_{\mathbb{R}^2} 
 (1-\chi_\lambda(x)) [P_\mu e^{-it\Delta}](x,y) P_{\mu}\Lambda(\phi)(t)(y) dy dt \label{intker}
 \end{align}
where for \eqref{intker} we use the integral kernel of $P_\mu e^{-i t \Delta}$.

Next, for the portion of $\phi_{\geq \lambda}$ outside of the ball $\{ |x| \geq \lambda^{-1} \}$,
we split into incoming and outgoing waves propagating backward and forward in time (respectively):
\begin{align*}
\chi_\lambda(x)\phi_{\geq \lambda}(0) 
=& \lim_{T\rightarrow \infty} i\int_0^T   \chi_\lambda(x) e^{-it\Delta} P_{\geq \lambda}\Lambda(\phi)(t)dt 
      \\
=&\;  i\int_0^\delta \chi_\lambda(x)P^{+}e^{-it\Delta} P_{\geq \lambda}\Lambda(\phi)(t)dt  
- i\int_{-\delta}^0  \chi_\lambda(x) P^{-}e^{-it\Delta} P_{\geq \lambda}\Lambda(\phi)(t)dt 
\\      
& +\lim_{T\rightarrow \infty} \sum_{\mu \geq \lambda} i\int_\delta^T 
\int_{\mathbb{R}^2} \chi_\lambda(x)[P^{+}_\mu e^{-it\Delta}](x,y) P_{\mu }\Lambda(\phi)(t)(y) dy dt     
\\
& -\lim_{T\rightarrow \infty} \sum_{\mu \geq \lambda} i\int_{-T}^{-\delta} 
\int_{\mathbb{R}^2} \chi_\lambda(x)[P^{-}_\mu e^{-it\Delta}](x,y) P_{\mu }\Lambda(\phi)(t)(y) dy dt 
\end{align*}            
As explained in \cite{KiViZh08}, this is to be interpreted as a weak $L^2$ limit, and we have 
\[
f_T\rightarrow f ~ weakly \Longrightarrow \|f\|\leq \limsup\|f_T\|
\]
The main point in the above two formulas is to cut our estimate into four different regions, according to whether we are in short/long time intervals and whether we are inside/outside the ball $\{ |x|\leq \lambda^{-1} \}.$
  
The following short time estimate works for any spatial region.            
\begin{prop}[Short-time estimate] \label{prop:ste}
Given any $\eta>0$ we can find some $\delta=\delta(\phi,\eta)>0$ such that 
\[
\| \int_0^\delta e^{-it\Delta} P_{\geq \lambda}\Lambda(\phi)(t)dt \|_{L^2}\leq \eta \mathcal{M}(\frac{\lambda}{8})
\]
provided $\lambda$ is large enough. 

Similar estimates hold on the time interval $[-\delta, 0]$ and for incoming/outgoing waves
under premultiplication by $\chi_{\lambda} P^{\pm}$.
\end{prop}
 
The proof is similar to that of \cite[Lemma 7.3]{KiTaVi09}, the main difference being that we must use the nonlocal
H\"older estimate \eqref{Holder1} and the estimate \eqref{Holder2}. As in the proof of extra regularity for the self-similar
case, we also use the fact that a high frequency output of $\Lambda(\phi)$ implies that there is a high frequency input term.
The details of how to perform the decomposition and apply \eqref{Holder1} and \eqref{Holder2} are performed similarly, and so we
omit the proofs.
  
To work with the long-time estimate, we notice that the integral kernels 
$P_\mu e^{-it\Delta}$ and $P^{\pm}_\mu e^{-it\Delta}$ have a stationary point when $ |x| -|y|\sim \mu|t|$. 
Hence we divide the region into $|y|\gtrsim \mu |t|$ and $|y|\ll \mu |t|$.

Take $\chi_{k}$ to be the characteristic function for 
\[
\{(t,y)| 2^k\delta \leq |t| \leq 2^{k+1}\delta, |y|\gtrsim \mu |t|\}
\]
\begin{prop}[Long-time estimate: main contribution] 
Let $0<\eta<1$ and $\delta $ be as in Proposition \ref{prop:ste}.
Then
\[ 
\sum_{\mu \geq \lambda} \sum_{k} \|\int_\delta^T 
\int_{\mathbb{R}^2} [P_\mu e^{-it\Delta}](x,y)  \chi_k(t,y)P_{\mu}\Lambda(\phi)(t)(y) dy dt
\|_{L^2_x} \lesssim \eta \mathcal{M}(\frac{\lambda}{8})
\]
for $\lambda$ large enough. A similar estimate holds under
premultiplication by $\chi_{k} P^{\pm}$.
\end{prop}
Now we just need to estimate the tails coming from the region 
\[
\{(t,y)| 2^k\delta \leq |t| \leq 2^{k+1}\delta, |y|\ll \mu |t|\}
\]
Since this is the non-stationary region, the kernels have better decay.
Let $\tilde{\chi}_k$ denote the characteristic function of this region. Then we have the following tail estimate. 
\begin{prop}[Long-time estimate: tails] \label{prop:ltet}
Let $0<\eta<1$ and $\delta $ be as in Proposition \ref{prop:ste}.
Then
\[ 
\sum_{\mu \geq \lambda} \sum_{k} \|\int_\delta^T 
\int_{\mathbb{R}^2} [P_\mu e^{-it\Delta}](x,y)  \tilde{\chi}_k(t,y)P_{\mu }\Lambda(\phi)(t)(y) dy dt
\|_{L^2_x} \lesssim \eta \mathcal{M}(\frac{\lambda}{8})
\]
for $\lambda$ large enough. A similar estimate holds under premultiplication $\tilde{\chi}_k P^{\pm}$.
\end{prop}
Together Propositions \ref{prop:ste}--\ref{prop:ltet} establish \eqref{sufflem}.

\section{Virial and Morawetz identities} \label{sec:moravir}

We recall
\[
F_{0r} = - \frac1r \Im(\bar{\phi} D_\theta \phi), \quad
F_{0 \theta} = r \Im(\bar{\phi} D_r \phi), \quad
F_{r \theta} = - \frac12 r |\phi|^2
\]
Because $dF = d^2 A = 0$, we have
\begin{equation} \label{dF}
\partial_t F_{r \theta} - \partial_r F_{0 \theta} + \partial_\theta F_{0r} = 0
\end{equation}
To rewrite this in terms of a natural stress-energy tensor, let
\[
T_{00} = \frac12 r |\phi|^2, \quad
T_{0 r} = r \Im(\bar{\phi} D_r \phi), \quad
T_{0 \theta} = \frac1r \Im(\bar{\phi} D_\theta \phi)
\]
Then \eqref{dF} may be rewritten as $\partial_\alpha T_{0 \alpha} = 0$.

\begin{lem}
We have
\begin{equation} \label{T0rt}
\begin{split}
\partial_t T_{0r}
=&
-(2 + 2r \partial_r) |D_r \phi|^2 + \frac12 r g \partial_r |\phi|^4 
\\&
+ \frac{1}{r} \partial_r |D_\theta \phi|^2 - \frac2r \partial_\theta \Re(\overline{D_\theta \phi} D_r \phi)
\\&
+ r \partial_r \left[ \frac{1}{r^2} \left( \frac12 \partial_\theta^2 |\phi|^2 - |D_\theta \phi |^2 \right) \right]
+ \left( \frac12 r \partial_r^3 + \frac12 \partial_r^2 - \frac{1}{2r} \partial_r \right) |\phi|^2
\end{split}
\end{equation}
\end{lem}
\begin{proof}
We write
\[
\begin{split}
\partial_t T_{0r} &=
r \Im(\overline{D_t \phi} D_r \phi) + r \Im(\bar{\phi} D_t D_r \phi) 
\\
&=
r \Im(\overline{D_t \phi} D_r \phi) + r \Im(\bar{\phi} D_r D_t \phi)  + r F_{0r} |\phi|^2
\\
&=
r \Im(\overline{D_t \phi} D_r \phi) + r \Im(\bar{\phi} D_r D_t \phi)  + 2 F_{\theta r} F_{0r}
\end{split}
\]
and calculate each piece separately, using \eqref{phievo}.

For the first term, we get
\[
\begin{split}
r \Im(\overline{D_t \phi} D_r \phi)
&=
-r \Re(\overline{D_r^2 \phi} D_r \phi) - |D_r \phi|^2 - \frac{1}{r} \Re(\overline{D_\theta^2 \phi} D_r \phi)
- r g |\phi|^2 \Re(\bar{\phi} D_r \phi)
\\
&=
-(1 + \frac12 r \partial_r) |D_r \phi|^2 - \frac14 g r \partial_r |\phi|^4
- \frac{1}{r} \Re(\overline{D_\theta^2 \phi} D_r \phi)
\end{split}
\]
Now
\[
\begin{split}
\Re(\overline{D_\theta^2 \phi} D_r \phi)
&=
\partial_\theta \Re(\overline{D_\theta \phi} D_r \phi)
- \Re(\overline{D_\theta \phi} D_r D_\theta \phi) - \Re(\overline{D_\theta \phi} i F_{\theta r} \phi)
\\
&=
\partial_\theta \Re(\overline{D_\theta \phi} D_r \phi)
- F_{\theta r} \Im(\bar{\phi} D_\theta \phi) - \frac12 \partial_r |D_\theta \phi|^2
\end{split}
\]
and so we can rewrite the first term as
\[
r \Im(\overline{D_t \phi} D_r \phi)
=
-(1 + \frac12 r \partial_r) |D_r \phi|^2 - \frac14 g r \partial_r |\phi|^4
- \frac1r \partial_\theta \Re(\overline{D_\theta \phi} D_r \phi)
- F_{\theta r} F_{0r}
+ \frac{1}{2r} \partial_r |D_\theta \phi|^2
\]

For the second term, we get
\[
\begin{split}
r \Im(\bar{\phi} D_r D_t \phi)
&=
r \Re(\bar{\phi} D_r^3 \phi) + r \Re(\bar{\phi} D_r (\frac1r D_r \phi))
+ r \Re(\bar{\phi} D_r (\frac{1}{r^2} D_\theta^2 \phi) ) + r g \Re(\bar{\phi} D_r (|\phi|^2 \phi))
\end{split}
\]
Now
\[
\begin{split}
r \Re(\bar{\phi} D_r^3 \phi) 
&= 
\frac12 r \partial_r^3 |\phi|^2
- \frac32 r \partial_r |D_r \phi|^2
\\
r \Re(\bar{\phi} D_r (\frac1r D_r \phi))
&=
- |D_r \phi|^2
+ \left( \frac12 \partial_r^2 - \frac{1}{2r} \partial_r \right) |\phi|^2
\\
r \Re(\bar{\phi} D_r (\frac{1}{r^2} D_\theta^2 \phi) )
&=
r \partial_r \left[ \frac{1}{r^2} \left( \frac12 \partial_\theta^2 |\phi|^2 - |D_\theta \phi |^2 \right) \right]
- \frac1r \Re(\overline{D_r \phi} D_\theta^2 \phi)
\\
r g \Re(\bar{\phi} D_r (|\phi|^2 \phi))
&=
\frac34 r g \partial_r |\phi|^4
\end{split}
\]
Hence
\[
\begin{split}
r \Im(\bar{\phi} D_r D_t \phi)
=&
- \left(1 + \frac32 r \partial_r \right) |D_r \phi|^2
+ \left( \frac12 r \partial_r^3 + \frac{1}{2} \partial_r^2 - \frac{1}{2r} \partial_r \right) |\phi|^2
\\
&
+ \frac{1}{2r} \partial_r |D_\theta \phi|^2
+ r \partial_r \left[ \frac{1}{r^2} \left( \frac12 \partial_\theta^2 |\phi|^2 - |D_\theta \phi |^2 \right) \right]- \frac1r \partial_\theta \Re(\overline{D_\theta \phi} D_r \phi)
\\
&- F_{\theta r} F_{0r} 
+ \frac34 r g \partial_r |\phi|^4
\end{split}
\]
Combining the above pieces yields \eqref{T0rt}.
\end{proof}

\begin{lem}[Virial and Morawetz identities]
A direct calculation relying upon integration by parts verifies the virial identity
\begin{equation} \label{virial}
\partial_t^2  \iint
r^2 T_{00} dr d\theta
= 4 \iint
\left( |D_r \phi|^2 + \frac{1}{r^2} |D_\theta \phi|^2 - \frac{g}{2} |\phi|^4 \right) r dr d\theta
\end{equation}
and the Morawetz identity
\begin{equation} \label{Morawetz}
\partial_t^2 \iint
r T_{00} dr d\theta
=
\frac12 \iint
\left( \frac{1}{r^2} |\phi|^2 - g |\phi|^4 \right) dr d\theta
\end{equation} 
\end{lem}
\begin{proof}
To prove \eqref{virial}, start with
\[
\partial_t^2 \iint r^2 T_{00} dr d\theta
=
- \partial_t \iint r^2
\left( \partial_r T_{0r} + \partial_\theta T_{0 \theta} \right) dr d\theta
=
2 \iint r \partial_t T_{0r} dr d\theta
\]
Then invoke \eqref{T0rt} to conclude
\[
\iint r \partial_t T_{0r} dr d\theta
=
2 \iint
\left( |D_r \phi|^2 + \frac{1}{r^2} |D_\theta \phi|^2 - \frac{g}{2} |\phi|^4 \right) r dr d\theta
=
4 E(\phi)
\]
To obtain \eqref{Morawetz}, write
\[
\partial_t^2 \iint r 
T_{00} dr d\theta
=
\iint 
\partial_t T_{0r} dr d\theta
\]
and then use \eqref{T0rt}.
\end{proof}

\begin{rem}
Under the equivariant ansatz, the components of the stress-energy tensor are radial, so that,
in particular the integrands of \eqref{virial} and \eqref{Morawetz} are independent of $\theta$.
Under this ansatz, the identity \eqref{T0rt} admits the simplification
\begin{equation} \label{Tortsimp}
\partial_t T_{0r}
=
-(2 + 2r \partial_r) |D_r \phi|^2 + \frac12 r g \partial_r |\phi|^4 
+ \frac{1}{r} \partial_r |D_\theta \phi|^2
- r \partial_r \left( \frac{1}{r^2} |D_\theta \phi|^2 \right)
+ \left( \frac12 r \partial_r^3 + \frac12 \partial_r^2 - \frac{1}{2r} \partial_r \right) |\phi|^2
\end{equation}
\end{rem}

\section{Absence of almost periodic solutions} \label{sec:noaps}

\begin{prop} 
Let $m \in \Z$ and let
$\phi \in H^1_m$ be a nontrivial solution of \eqref{equiCSS} with $g < 1$. Then $E(\phi) > 0$.
\end{prop}
\begin{proof}
The main tool required is the so-called Bogomol'nyi identity, which states
\begin{equation} \label{Bogo}
|D_x \phi|^2 = |D_+ \phi|^2 + \nabla \times J - F_{12} |\phi|^2
\end{equation}
where $D_{\pm} := D_1 \pm i D_2$ and $J = (J_1, J_2)$ with $J_k := \Im(\bar{\phi} D_k \phi)$.
This identity can be motivated by the factorization
\[
D_j D_j \phi = (D_1 - i D_2)(D_1 + i D_2) \phi + F_{12} \phi
\]
and both can be verified by direct calculation.
Using \eqref{Bogo} and Green's theorem, we obtain
\begin{equation} \label{enerbogo}
E(\phi) := \frac12 \int_{\R^2} \left[ |D_x \phi|^2 - \frac{g}{2} |\phi|^4 \right] dx
=
\frac12 \int_{\R^2} \left[ |D_+ \phi|^2 + \frac12 (1 - g) |\phi|^4 \right] dx
\end{equation}
From this we conclude that if $g < 1$ and $\phi$ is not zero a.e., then
$E(\phi) > 0$.
\end{proof}

\subsection{Ruling out the self-similar scenario}
As a corollary of \eqref{ExtraRsscon}, used to prove Theorem \ref{ExtraRss}, 
we obtain the following.
\begin{cor}
Let $g < 1$. Critical equivariant self-similar solutions do not exist.
\end{cor}
\begin{proof}
For any $s \geq 0$,
\[
\sup_{t \in (0, \infty)} \int_{|\xi| > A t^{-1/2}} |\hat{\phi}(\xi, t)|^2 d \xi
\leq
C_s A^{-s},
\quad
A > A_0(s)
\]
Therefore
\begin{equation} \label{dotH}
\| \phi(t, \cdot) \|_{\dot{H}^s(\R^2)} \lesssim t^{- \frac{s}{2}} = [N(t)]^s
\end{equation}
Thanks to the following lemma, taking $t \to \infty$ in \eqref{dotH} 
implies that the conserved energy $E(\phi)$ must be zero and hence the solution
$\phi$ trivial.
\end{proof}

\begin{lem}
Let $m \in \Z$ and
let $\phi \in L^\infty_t L^2_m$ be a solution of \eqref{equiCSS}
with $E(\phi)$, given by \eqref{energy}, finite.
Then
\begin{equation} \label{energy-H1}
|E(\phi)| \lesssim \| \phi \|_{\dot{H}^1}
\end{equation}
where the constant depends upon $g$ and the charge $\chg(\phi)$.
\end{lem}
\begin{proof}
First we note that
\[
|D_x \phi|^2 \lesssim |\nabla \phi|^2 + |A_x \phi|^2
\]
To control $A_x \phi$ in $L^2$, use $|A_j| = \frac1r |A_\theta| \lesssim \| \phi \|_{L^4}^2$,
where the last inequality follows from \eqref{Athet2}.
Therefore
\[
\| D_x \phi \|_{L^2}^2 \lesssim \| \nabla \phi \|_{L^2}^2 + \| \phi \|_{L^4}^2 \| \phi \|_{L^2}^2
\]
The lemma now follows from the Gagliardo-Nirenberg inequality
\[
\| f \|_{L^4}^4 \lesssim \| \nabla f \|_{L^2}^2 \| f \|_{L^2}^2
\]
\end{proof}

\subsection{Ruling out global almost periodic solutions}

Let $\chi : \R_+ \to [0, 1]$ be a smooth cut-off function equal to one on $[0, 1]$ and zero on $[2, \infty)$.
For any given $R > 0$, define $\chi_R(r) := \chi(r / R)$. Set
\[
I_R(\phi) := \int_0^\infty T_{0r} \chi_R r dr
\]

\begin{lem}[Localized virial identity]
Let $m \in \Z$ and $\phi \in L^\infty H^1_m$. Then
\begin{equation} \label{viriallocal}
\begin{split}
\frac{d}{dt} I_R(\phi) =& \; 4 E(\phi) \\
&+ 2 \int_0^\infty \left( |D_r \phi|^2 + \frac{1}{r^2} |D_\theta \phi|^2 - \frac{g}{2} |\phi|^4 \right) (\chi_R - 1) r dr \\
&+ 2 \int_0^\infty \left( |D_r \phi|^2 - \frac{5}{4} \frac{|\phi|^2}{r^2} - \frac{g}{4} | \phi |^4 \right) r \chi_R^\prime r dr \\
&- \frac72\int_0^\infty \frac{|\phi|^2}{r^2} r^2 \chi_R^{\prime \prime} rdr
- \frac12 \int_0^\infty \frac{|\phi|^2}{r^2}  r^3 \chi_R^{\prime \prime \prime} r dr
\end{split}
\end{equation}
\end{lem}
\begin{proof}
This follows from using \eqref{Tortsimp} and integrating by parts.
\end{proof}

\begin{cor} \label{cor:globap}
Let $g < 1$. Global equivariant critical elements do not exist.
\end{cor}
\begin{proof}
Invoking Theorem \ref{ExtraR}, we have that for each $s \geq 0$ the estimate
\[
\| \phi(t, \cdot) \|_{\dot{H}^s(\R^2)} \leq C_s
\]
holds uniformly in time.
Next, let $\eta > 0$ and take $R = 2 C(\eta)$ so that
\[
\int_{|x| > R / 2} |\phi(t, x)|^2 dx < \eta
\]
for all time.
By interpolating, we can control the energy far from the origin:
\[
\int_R^\infty \left( |D_r \phi|^2 + \frac{1}{r^2} |D_\theta \phi|^2 - \frac{g}{2} |\phi|^4 \right) r dr \lesssim \eta^\frac12
\]
Using this in \eqref{viriallocal} implies
\[
\frac{d}{dt} I_R(\phi) \geq 4 E(\phi) - C \eta^\frac12
\]
Therefore, by conservation of energy, we have for $\eta$ sufficiently small that
\begin{equation} \label{growth}
\frac{d}{dt} I_R(\phi) \gtrsim 1
\end{equation}
On the other hand, by \eqref{viriallocal} and \eqref{energy-H1},
\[
| I_R(\phi) | \lesssim R \| \phi \|_{L^2} \| \phi \|_{\dot{H}^1} \lesssim R C_1
\]
holds uniformly in time. 
This contradicts \eqref{growth} for $t$ sufficiently large.
\end{proof}

\section{The focusing problem} \label{sec:ext}

In the focusing problem we shall restrict ourselves to $m \geq 0$. 
This is the physically interesting case for \eqref{equiCSS} as written. 
In fact, the natural Chern-Simons-Schr\"odinger
system for $m < 0$ is not \eqref{equiCSS}, but rather an analogous one with the signs
in the field constraints flipped. For further discussion of this point, see \cite[II.~C., E.]{Du95}.

\subsection{The case $g = 1$.}

\begin{lem} \label{lem:g1}
Let $g = 1$ and $m \in \Z_+$.
Suppose that $\phi \in L^\infty_t H^1_m$ is a solution of \eqref{equiCSS} 
with $E(\phi) = 0$. Then $\phi$ is a soliton.
\end{lem}
\begin{proof}
Straightforward calculations reveal
\begin{equation} \label{Dplus}
D_+ = e^{i \theta} \left( D_r + \frac{i}{r} D_\theta \right)
\end{equation}
and
\[
|D_\theta \phi|^2 = (m + A_\theta)^2 |\phi|^2
\]
By \eqref{enerbogo}, $E(\phi) = 0$ implies $D_+ \phi = 0$ a.e. 
For $m$-equivariant $\phi$, this implies
\[
\partial_r \phi = \frac{1}{r} (m + A_\theta) \phi
\]
Consequently,
\[
\frac12 \partial_r |\phi|^2 = \frac{1}{r}(m + A_\theta) |\phi|^2 = \partial_r A_0
\]
so that
$A_0 = \frac12|\phi|^2$. Therefore $\phi$ is an equivariant solution of the self-dual 
Chern-Simons-Schr\"odinger system
\begin{equation} \label{SDCSS}
\begin{cases}
(D_1 + i D_2) \phi &= 0 \\
A_0 &= \frac12 |\phi|^2 \\
\partial_1 A_2 - \partial_2 A_1 &= - \frac12 |\phi|^2
\end{cases}
\end{equation}
Such solutions constitute static solutions to \eqref{equiCSS} (with $g = 1$).
Conversely, any $H^1$ static solution of \eqref{equiCSS} with $g = 1$ has $E(\phi) = 0$
(for a short proof, see \cite{HuSe13}).
\end{proof}
If $m \in \Z_+$, then explicit equivariant solutions are given by
\[
\begin{cases}
\phi^{(m)}(t, x) &= \sqrt{8} \lambda (m + 1) \frac{|\lambda x|^m}{1 + |\lambda x|^{2(m+1)}} e^{i m \theta} \\
A_j^{(m)}(t, x) &= 2(m+1) \lambda^2 \frac{\epsilon_{jk} x_k |\lambda x|^{2m}}{1 + |\lambda x|^{2(m+1)}} \\
A_0^{(m)}(t, x) &= 4 \left[ \frac{\lambda (m+1) |\lambda x|^m}{1 + |\lambda x|^{2(m+1)}} \right]^2 
\end{cases}
\]
where $\lambda > 0$ is a free scaling parameter and $\epsilon_{jk}$ is the anti symmetric tensor with
$\epsilon_{12} = 1$. Such solutions are discussed, for instance, in \cite{JaPi90, JaPi91}.
For any value $\lambda > 0$, we find
\[
\chg(\phi^{(m)}) = 8 \pi (m + 1)
\]
Uniqueness of these explicit soliton solutions is discussed in \cite{JaPi90} and a proof
can be given by combining the arguments of \cite{ChWa94} with the equivariance ansatz.

With Lemma \ref{lem:g1} in hand, we can conclude the proof of Theorem \ref{thm:main2} using arguments from
\S \ref{sec:noaps}.



\subsection{The case $g > 1$.}

\begin{lem} \label{lem:focus}
Let $g > 1$. Then there exists a constant $c_{g} > 0$ such that
any nontrivial $H^1$ solution $\phi$ of \eqref{equiCSS} with $E(\phi) \leq 0$
satisfies $\chg(\phi) \geq c_{g}$.
\end{lem}
\begin{proof}
Using \eqref{energy} we see that $E(\phi) \leq 0$ implies
\[
\frac{2}{g} \| D_x \phi \|_{L^2_x}^2 \leq \| \phi \|_{L^4_x}^4
\]
We can combine this with the covariant Gagliardo-Nirenberg inequality 
(e.g., see \cite[(2.28)]{BeBoSa95} for a proof), which states that
\begin{equation} \label{CovGN}
\| \phi \|_{L^4_x}^4 \lesssim \| D_x \phi \|_{L^2_x}^2 \| \phi \|_{L^2_x}^2,
\end{equation}
to conclude that $E(\phi) \leq 0$ implies
\[
\| \phi \|_{L^2_x}^2 \gtrsim \frac{2}{g}
\]
\end{proof}
Using the Bogomol'nyi identity \eqref{Bogo}, we may arrive at the following inequality, which
is similar to an inequality of Byeon, Huh, and Seok \cite[p.~1607]{ByHuSe12}.
\begin{lem}
Let $\phi$ be an $H^1$ solution of \eqref{CSS}, $g \in \R$. Then
\begin{equation} \label{GNtype1}
\| \phi \|_{L^4_x}^4 \leq 4 \| D_r \phi \|_{L^2_x} \| \frac1r D_\theta \phi \|_{L^2_x}
\end{equation}
\end{lem}
\begin{proof}
Integrating \eqref{Bogo} over $\R^2$ and using the observation \eqref{Dplus},
we conclude
\begin{equation} \label{L4eq}
\int \frac12 |\phi|^4
=
2 \int \Im(\overline{r^{-1} D_\theta \phi} D_r \phi)
\end{equation}
Then \eqref{GNtype1} follows by Cauchy-Schwarz.
\end{proof}
Applying Young's inequality to \eqref{GNtype1} yields
\begin{equation} \label{CovSobo}
\| \phi \|_{L^4_x}^4 \leq 2 \| D_r \phi \|_{L^2_x}^2 + 2 \| \frac1r D_\theta \phi \|_{L^2_x}^2
= 2 \| D_x \phi \|_{L^2_x}^2
\end{equation}
In particular, in the $g = 1$ case, zero-energy solutions of \eqref{equiCSS} are precisely
those that yield equality in \eqref{CovSobo}.
More generally, using \eqref{energy} and \eqref{L4eq}, we observe that $E(\phi) \leq 0$ implies that $\phi$
satisfies the following reverse Cauchy-Schwarz inequality:
\begin{equation} \label{ReverseCS}
\int \left[ |\partial_r \phi|^2 + \frac{1}{r^2} |D_\theta \phi|^2 \right] dx
\leq
2 g
\int \Im(\overline{r^{-1}D_\theta \phi} \partial_r \phi) dx
\end{equation}
\begin{rem}
The constants in Lemma \ref{lem:focus} are universal (in that they are not dependent upon
the equivariance index $m$ or even upon the satisfaction of the equivariance ansatz) 
but not sharp. In the next lemma we show that, given an equivariance
index $m \in \Z_+$, the sharp charge threshold for the class $H^1_m$ may be found by
minimizing over nontrivial energy zero solutions in that class.
\end{rem}
\begin{lem} \label{lem:7.5}
Let $m \in \Z_+$ and, for $\phi \in H^1_m$,
\begin{equation} \label{Jenergy}
J(\phi)
:=
\int_{\R^2} \left[
|\partial_r \phi|^2 + \frac{1}{r^2} \left(m - \frac12 \int_0^r |\phi|^2 s ds \right)^2 |\phi|^2 - 
\frac{g}{2} |\phi|^4 \right] r dr
\end{equation}
Then
\[
\inf_{0 \neq \phi \in H^1_m: J(\phi) \leq 0} \chg(\phi) = \inf_{0 \neq \phi \in H^1_m: J(\phi) = 0} \chg(\phi)
\]
\end{lem}
\begin{proof}
Note that for solutions $\phi \in H^1_m$ of \eqref{equiCSS}, the expression for $J(\phi)$
is equivalent to that of $E(\phi)$, with $E$ defined as in \eqref{energy}.
Let $\phi_n \in H^1_m$ for $n = 1, 2, 3, \ldots$ be a minimizing sequence
with $\chg(\phi_n) \to I$, where
\[
I := \inf_{0 \neq \phi \in H^1_m: J(\phi) \leq 0} \chg(\phi)
\]
Then, for $\alpha \in \R$,
\[
J(\alpha \phi_n) = 
\alpha^2 \int_{\R^2} \left[ |\partial_r \phi_n|^2 + \frac{1}{r^2}
\left(m - \alpha^2 \frac12 \int_0^r |\phi_n|^2 s ds \right)^2 |\phi_n|^2
- \alpha^2 \frac{g}{2} |\phi_n|^4 \right] dx
\]
Because
\[
\lim_{\alpha \to 0} \alpha^{-2} J(\alpha \phi_n)
=
\int_{\R^2} \left[ |\partial_r \phi_n|^2 + \frac{m}{r^2} |\phi_n|^2 \right] dx > 0,
\]
there exists $\alpha_n \in (0, 1]$ such that $E(\alpha_n \phi_n) = 0$.
Set $\psi_n := \alpha_n \phi_n$. Then $E(\psi_n) = 0$ and $\chg(\psi_n) \leq \chg(\phi_n)$.
Passing to a convergent subsequence, we obtain
\[
\lim \chg(\psi_n) \leq \lim \chg(\phi_n) = I
\]
\end{proof}
\begin{lem} \label{lem:7.6}
The set of minimizers over $J(\phi) = 0$ is nonempty.
\end{lem}
The proof for $m = 0$ is found in \cite[\S 5]{ByHuSe12}
and generalizes to the case $m > 0$. There is a lack of compactness due to the scaling symmetry
which is removed by renormalizing the $\dot{H}^1$ norm. Once this is done, one may pass
to a weak limit in $H^1_m$ that also converges strongly in $L^q_m$, $q > 2$.

In the next lemma we characterize energy-zero minimal charge solutions $\phi \in H^1_m \setminus \{0\}$.
\begin{lem} \label{lem:7.7}
Let $\phi \in H^1_m \setminus \{0\}$ with $E(\phi) = 0$ 
have minimal charge among all such nontrivial $m$-equivariant functions with energy zero.
Then there exists $\lambda \in \R$ such that
$\psi(t, x) := e^{i \lambda t} \phi(x)$ is a weak
solution of \eqref{equiCSS}.
\end{lem}
\begin{proof}
We use Lagrange multipliers, which necessitates taking the first variation of $J(\phi)$.
Varying the $\phi(r, \theta)$ terms leads to 
\[
 2 \int \left[
 \Re(\overline{\partial_r \phi} \partial_r \psi) + \frac{1}{r^2} (m + A_\theta)^2 \Re(\bar{\phi} \psi)
- g |\phi|^2 \Re(\bar{\phi} \psi) \right] r dr
\]
which upon integration by parts becomes
\[
- \int \Re\left( \bar{\psi} (\partial_r^2 + r^{-1} \partial_r + r^{-2} D_\theta + g |\phi|^2) \phi \right) r dr
\]
The additional contribution from the variation of $\phi(s, \rho)$ is
\begin{equation} \label{addcont}
- \int \left( \frac{2m}{r^2} - \frac{1}{r^2} \int_0^r |\phi|^s s ds \right) \int_0^r \Re(\bar{\phi} \psi) s ds |\phi|^2 r dr
\end{equation}
Let
\[
F(r) = - \int_r^{\infty} \left( \frac{m}{r} - \frac{1}{2r} \int_0^r |\phi|^2 s ds \right) |\phi|^2 dr
\]
Then \eqref{addcont} may be rewritten as
\[
-2 \int \int_0^r \Re(\bar{\phi} \psi) s ds \partial_r F(r) dr
\]
which upon integration by parts is seen to be
\[
2 \int_0^\infty F(r) \Re(\bar{\phi} \psi) dr
\]
With the observation that in fact we may take $A_0 = F(r)$, the proof is complete.
\end{proof}
The above variation is discussed in \cite[II.~B.]{Du95} and is similar to 
the approach of \cite{ByHuSe12}.
Solutions $\psi \in L^\infty H^1_m$ of \eqref{equiCSS} of the form 
$\psi(t, x) = e^{i \lambda t} \phi$, $\phi \in H^1_m$, we call standing wave solutions.
\begin{rem}
When $g = 1$, static solutions (standing wave solutions with $\lambda = 0$) exist
but  $\lambda \neq 0$ standing wave solutions do not.
When $g > 1$ and $m = 0$, Byeon, Huh, and Seok \cite[Rem.~5.1]{ByHuSe12}
conjecture that there are no static solutions.
\end{rem}
\begin{rem}
Together Lemmas \ref{lem:7.5}, \ref{lem:7.6}, and \ref{lem:7.7} complete the characterization
of the constants $c_{m, g}$ claimed in Theorem \ref{thm:main3}. Combining this with
the arguments of \S \ref{sec:noaps} completes its proof.
\end{rem}
We conclude with two Pohozaev-type identities of independent interest.
\begin{lem}[Pohozaev identity]
Let $\phi \in L^\infty_t H^1_m$ be a standing wave solution of \eqref{equiCSS}.
Then
\begin{equation} \label{Pohozaev}
\int (\lambda + A_0) |\phi|^2 dx = \frac{g}{2} \int |\phi|^4 dx
\end{equation}
\end{lem}
\begin{proof}
For a standing wave with frequency $\lambda$ we can write
\[
\int (\lambda + A_0) |\phi|^2 dx
=
\int \Im(\bar{\phi} D_t \phi) dx
\]
Next we replace $D_t \phi$ using the first equation of \eqref{CSScom} and integrate by parts.
\end{proof}
Through integrating by parts (differently) in \eqref{Pohozaev},
we can recover the Pohozaev-type identity established in the case $m = 0$
in \cite[Prop.~2.3]{ByHuSe12}.
\begin{cor}
Let $\phi \in L^\infty_t H^1_m$ be a standing wave solution of \eqref{equiCSS}. Then
\[
(\lambda - 2 m A_0(0)) \int |\phi|^2 dx
+ 2 \int_0^\infty \frac{1}{r^2} |D_\theta \phi|^2 dx = \frac{g}{2} \int |\phi|^4 dx
\]
\end{cor}
\begin{proof}
We have
\[
\int_0^\infty A_0 |\phi|^2 r dr
=
\int_0^\infty \left( - \int_r^\infty \frac{m + A_\theta}{s} |\phi|^2(s) ds \right) |\phi|^2(r) r dr
\]
Now write
\[
|\phi|^2 r = -2 \partial_r \left( m - \frac12 \int_0^r |\phi|^2 s ds \right)
\]
Integrating by parts yields
\[
\int_0^\infty A_0 |\phi|^2 r dr
=
-2m A_0(0) + 2 \int_0^\infty \frac{(m+ A_\theta)^2}{r^2} |\phi|^2 r dr
\]
\end{proof}

\bibliography{CSSbib}
\bibliographystyle{amsplain}

\end{document}